\documentclass[smallextended]{svjour3}
\smartqed 

\usepackage{amsmath} %spacing of integrals..
\usepackage{graphicx}
\usepackage{amsfonts}
\usepackage{dsfont}
\usepackage{algorithm}
\usepackage{algorithmic}

\newcommand{\R}{\mathbf{R}}
\newtheorem{assumption}{Assumption}
\newcommand{\reff}[1]{{\rm(\ref{#1})}}
\newcommand{\1}{\mathds{1}}
\newcommand{\Ac}{\mathcal{A}}
\newcommand{\E}{\mathbb{E}}
\newcommand{\Sup}{\displaystyle\sup}
\newcommand{\Frac}{\displaystyle\frac}
\newcommand{\supp}{\operatorname{supp}} 

\title{Adaptive sparse grids for time dependent Hamilton-Jacobi-Bellman equations in stochastic control}
\titlerunning{Sparse grids for Hamilton-Jacobi-Bellman equations}
\author{Xavier Warin }
\institute{X. Warin \at EDF R\&D \& FiME, Laboratoire de Finance des March\'es de l'Energie (www.fime-lab.org)\\
           Tel: +33-1-47654184\\
           \email{xavier.warin@edf.fr}
           }

\begin{document}
\maketitle
\begin{abstract} 
We introduce some sparse grids interpolations used in Semi-Lagrangian schemes for linear and fully non-linear diffusion Hamilton Jacobi Bellman equations arising in stochastic control. 
We  prove that the method introduced converges toward the viscosity solution of the problem and we show that some potentially high order schemes can be efficiently implemented. Numerical test in dimension $2$ to $5$ are achieved and show that deterministic methods can be used efficiently in stochastic control  in moderate dimension.
\end{abstract}

\section{Introduction}
\label{intro}
We are interested in a classical stochastic control problem whose value function is solution of the following Hamilton Jacobi equation:
\begin{eqnarray}
\frac{\partial v}{\partial t}(t,x) &  -&  \inf_{a \in \mathop{A}} \left( \frac{1}{2} tr(\sigma_a(t,x)\sigma_a(t,x)^T D^2 v(t,x)) + b_a(t,x) D v(t,x) \right.  \nonumber  \\
  &  & \left .  + c_a(t,x) v(t,x)+ f_a(t,x) \vphantom{\int_t} \right)  =  0  \mbox{ in  } \mathop{Y} \nonumber  \\
v(0,x)& = & g(x) \mbox{ in } \R^d
\label{hjb}
\end{eqnarray}
where $Y := [0,T] \times \R^d$, $\mathop{A}$ is a complete metric space.
$\sigma_a(t,x)$  is a $d \times q$ matrix   so  that $\sigma_a(t,x) \sigma_a(t,x)^T$ is a $d \times d$ symmetric matrix. The  $b_a$, $c_a$ and $f_a$ coefficients are functions defined on  $Y$ with values ​​respectively in $\R^d$,  $\R$ and $\R$.\\
Let's introduce an $\R^d$-valued controlled process $X^{x,t}_s$ defined on a filtered probability space $\left(\Omega,\mathcal{F},\mathbb{F},\mathbb{P}\right)$ by
\begin{eqnarray*}
dX^{x,t}_s & = & b_a(t,X^{x,t}_s) ds + \sigma_a(s,X^{x,t}_s) dW_s  \\
X^{x,t}_t &=& x 
\end{eqnarray*}
where $a$ is progressively measurable.
This kind of problem arises when  minimizing a cost function  $J(t,x) =  \mathbb{E}[\int_{t}^T f_a(s,X^{x,t}_s) e^{c_a(s,X^{x,t}_s)} ds + g(X^{x,t}_T)]$ with respect to the control $a$. It is well known \cite{soner} that the optimal value $ \hat J(t,x) = \inf_{a} J(t,x,a)$ is a viscosity solution of equation \reff{hjb}.\\
Modified finite difference schemes can be used to treat  this problem  \cite{bonnans1} approximating finite differences with non adjacent points in the stencil such that
the scheme is monotone according to  Barles Souganidis framework \cite{barles3}. Some stochastic approaches based on the resolution of a Second Order Backward Stochastic Differential Equation  have been developed in \cite{warin,xiaolu}. They can tackle high dimension problems (an example in dimension 5 has been explored in \cite{warin}).
Very recently,  it has been proved that the solution of  equation \reff{hjb} admits a probabilistic representation by means of a Backward Stochastic Differential
Equation with positive jumps \cite{pham1}. In this approach, the underlying controlled process  is replaced by a non controlled one and the space of controls is explored randomly by a pure jump process leading to a new algorithm based on regression  \cite{pham2}.\\
In the present article, we explore the case of Semi-Lagrangian methods based on the scheme developed by Camilli Falcone \cite{Camilli}. Some variant have been
developed by Munos Zidani \cite{Zidani} and  Debrabant Jakobsen  \cite{Jakobsen}. In this approach, the brownian motion is discretized taking two values ​​of the order of  $\sqrt{h}$. Starting from an initial grid discretization, the algorithm needs to interpolate the function at some points outside the grid.
Recently it has been proved in \cite{warin1,warin2} that the monotonicity of the scheme can be somewhat relaxed leading to locally high order schemes converging to the viscosity
solution of the problem. In fact, Semi-Lagrangian schemes are intrinsically monotone due to the time discretization of the problem.
The interpolation  only brings an error independent of the time discretization. This error can be controlled during the time step iterations if the interpolator
doesn't bring too many oscillations. The idea in  \cite{warin1,warin2} is to interpolate the solution  on an $hp$ finite element basis and truncate it when the interpolated solution oscillates too much. This approach used in a parallel framework  (see \cite{warin1} for details) permits to tackle problems in dimension 3 or 4 at most when the dimension of the control space is limited to one.\\
In order to treat moderate or  high dimension problems, the classical interpolation process is too costly and the number of discretization points grows up exponentially with  the dimension $d$ of the problem, being equal to $N^d$ where $N$ is the number of points in one dimension.  This ``curse of the dimension'' can be mitigated for rather smooth function by using sparse grid methods. 
The sparse grid method has been first introduced in \cite{zwenger} for partial differential equations and the idea can be traced in the Russian literature for quadrature and interpolation in \cite{smolyack}. Recently sparse grids have been used successfully in many problems in high  dimension (see the overview articles \cite{bungartz1,pfluger}).
 The main idea of the sparse grids consists in removing point of the full grid points that are not necessary for
a good representation
of regular functions. When the function is null at the boundary of a cube $[0,1]^d$, and noting $N$ the number of points in each direction of the corresponding full grid,
the number of points can be reduced to $O( N log(N)^{d-1})$ and the error in the infinite norm only deteriorates from $O(N^{-2})$ in the case of full grid to $O(N^{-2} \log(N)^{d-1})$ for the sparse case with linear interpolators.
Of course it can only be achieved for regular function. 
In fact it has been shown in \cite{bungartz2} that the rate of convergence of sparse grids using local hat function was
directly linked to the cross directives of the function interpolated. When the solution is not smooth (for example Lipschitz in our problem),  there is no hope
to get an accurate solution if simple sparse grids are used. In order to circumvent this problem, adaptive sparse grids have been developed in \cite{gerstner}. Using
an estimator of the error associated to the surplus of the hierarchical representation of the solution, effective adaptive methods have been developed to get accurate 
estimation of the solution.
It is to be noticed that this approach is mainly effective if the singularity is located  in a small number of dimensions. See for example \cite{ma} to get examples of the use of adaptive sparse grids in physics.
For the basket option in finance it is well known \cite{reisinger,osterlee}  that a change of coordinate has to be achieved  such that the sparse grids works by increasing
the number of points in the dimension where the singularity lies. \\
As for Hamilton Jacobi Bellman equation with first order derivatives, a first numerical study has been recently achieved with adaptive with sparse grids and linear hat functions 
\cite{bokanowski}.
In this present  paper we first recall the classical regularity results associated to the problem, the time discretization scheme and time convergence results associated.
In a second part, we present the sparse grids method.
We prove that if the adaptation is effective,  the solution calculated converges to the viscosity solution of the problem when a linear sparse interpolator is used. 
The rate of convergence of hat function is low (see \cite{warin1} for test of convergence) and we explain how to implement some potentially locally high order schemes converging towards the viscosity solution
by modifying the algorithms proposed in  \cite{bungartz3,bungartz4} and by using a truncation as explained in \cite{warin1,warin2}. 
As pointed out, on all our tests the truncation modifies only very slightly the solution. 
At last we numerically test the schemes developed and prove their efficiency.

\section{Regularity results and Semi Lagrangian Scheme}
\label{regularityRes}
For a bounded function $w$, we set
\begin{eqnarray*}
 | w|_0  = \sup_{(t,x) \in Y} | w(t,x)|, &  \quad  & [w]_1 = \sup_{(s,x) \ne (t,y)} \frac{ |w(s,x) -w(t,y)|}{|x-y| + |t-s|^{\frac{1}{2}}}
\end{eqnarray*}
and  $|w|_1 = | w|_0 +[w]_1 $. $C_b(Y)$ denote the space of bounded functions on  $Y$ and  $C_1(Y)$ will stand for the space of functions  with a finite  $|\quad |_1$ norm.\\
For  $t$ given, we denote
\begin{eqnarray*}
||w(t,.) ||_{\infty} = \sup_{x \in \R^d } | w(t,x)| 
\end{eqnarray*}
We use the classical assumption on the data of \reff{hjb} for a given $\hat K$: 
\begin{eqnarray}
 \sup_a |g|_1 + |\sigma_a|_1 + |b_a|_1 +  |f_a|_1  + | c_a|_1 \le \hat K
\label{coeff}
\end{eqnarray}

The following proposition \cite{Jakobsen} gives us the existence of a solution in the space of bounded Lipschitz functions
\begin{proposition}
If the coefficients of the equation  \reff{hjb} satisfy \reff{coeff}, there exists a unique viscosity solution of the equation \reff{hjb}  belonging to  $C_1(Y)$.  If $ u_1 $ and $u_2$ are  respectively sub and super-solution of equation  \reff{hjb} satisfying  $u_1(0,.) \le u_2(0,.)$  then  $u_1 \le u_2$.
\end{proposition} 
The equation \reff{hjb} is discretized in time by the time scheme proposed by Camilli Falcone \cite{Camilli}  for a time step $h$.
\begin{eqnarray}
v(t+h,x) & = &  \inf_{a \in \mathop{A}}  \left[ \sum_{i=1}^q \frac{1}{2q} (v(t, \phi^{+}_{a, h, i}(t,x)) +  v(t, \phi^{-}_{a, h, i}(t,x)) )
 +  f_a(t,x) h  \right. \nonumber \\ 
&  & \left. + c_a(t,x) h v(t,x) \vphantom{\int_t}  \right]  \nonumber \\
    & := & v(t,x)+ \inf_{a \in \mathop{A}}   L_{a,h}(v)(t,x)   \label{hjbCamilli}  
\end{eqnarray}

with
\begin{eqnarray*}
L_{a,h}(v)(t,x) & = & \sum_{i=1}^q \frac{1}{2q} (v(t, \phi^{+}_{a, h, i}(t,x)) + v(t, \phi^{-}_{a, h, i}(t,x)) -  2 v(t,x))  \nonumber \\
              &    & + h c_a(t,x) v(t,x) + h f_a(t,x)  \\
\phi^{+}_{a, h, i}(t,x) & =& x +b_a(t,x) h + (\sigma_a)_i(t,x) \sqrt{h q} \\
\phi^{-}_{a, h, i}(t,x) & =& x +b_a(t,x) h - (\sigma_a)_i(t,x) \sqrt{h q} 
\end{eqnarray*}
where  $(\sigma_a)_i$ is the $i$-th column of $\sigma_a$. We note that it is also possible to choose other types of discretization  as those defined
in \cite{Zidani}.\\
In order to define the solution  at each date,  a condition on the value chosen for $v$ between $0$ and $h$ is required. We choose a time linear interpolation once the solution has been calculated at date  $h$:
\begin{eqnarray}
v(t,x) = (1- \frac{t}{h})  g(x)+  \frac{t}{h} v(h,x), \forall t \in [0,h]. 
\label{hjbCamilli1}
\end{eqnarray}
We denote  $v_h$  the  discrete solution obtained for a discretization with a time step  $h$. Following \cite{Jakobsen}
\begin{proposition}
\label{converDisH}
  The solution $v_h$ of equations \reff{hjbCamilli} and \reff{hjbCamilli1} is uniquely defined and belongs to   $C^1(Y)$. We check that if $ h \le (16 \sup_a \left \{ |\sigma_a|_1^2 + | b_a|_1^2 +1 \right \} \wedge 2 \sup_a |c_a|_0)^{-1} $, there exists $C$ such that 
\begin{eqnarray*}
| v -v_h |_0 \le C h^{\frac{1}{4}} 
\end{eqnarray*}
Moreover, there exists  $C$ independent of $h$ such that 
\begin{eqnarray}
\label{LipschitVh}
| v_h|_0 & \le & C \\
|v_h(t,x)-v_h(t,y)| & \le & C |x-y|, \forall (x,y) \in Y^2
\end{eqnarray}
\end{proposition}
It is clear that if the scheme \reff{hjbCamilli} is used at some discretized point $x$  of a grid, then some interpolation are needed for the calculation of $L_{a,h}(v)(t,x)$.
We next develop the interpolation method based on sparse grids for different function basis.

\section{Classical Sparse Grids Interpolation}
\label{classicalSG}
We recall some classical results on sparse grids that can be found in \cite{pfluger}. 
\subsection{Linear case}
\label{linearSection}
We first assume that the function we interpolate is null at the boundary.
By a change of coordinate an hyper-cube domain can be changed to a domain $\omega = [0,1]^d$.
Introducing the hat function $\phi^{(L)}(x)= \max(1-|x|,0)$,  we obtain the following local one dimensional  hat function  by translation and dilatation  
$$\phi_{l,i}^{(L)}(x) = \phi^{(L)}(2^l x-i)$$
depending on the level $l$ and the index $i$, $0< i< 2^l$. 
The grid points used for interpolation are noted  $x_{l,i}=2^{-l}i$. 
In dimension $d$, we introduce the basis functions 
$$ \phi_{\underline l, \underline i}^{(L)}(x) = \prod_{j=1}^d \phi_{l_j,i_j}^{(L)}(x_j)$$
via a tensor approach  for a point $ \underline x = (x_1, ....x_d)$, a multi-level $ \underline l := (l_1, .., l_d)$ and a multi-index $\underline i := (i_1,.. , i_d).$ 
The grid points used for interpolation are thus $x_{\underline l,\underline i} := (x_{l_1,i_1},.., x_{l_d,i_d})$.\\
We next introduce the  index set $$B_{\underline l} := \left \{ \underline i: 1 \le i_j \le 2^{l_j}-1,  i_j \mbox{ odd }, 1 \le j \le d \right\}$$
 and the space of hierarchical basis
\begin{eqnarray*}
 W_{\underline l}^{(L)}&:=& span\left \{ \phi^{(L)}_{\underline l, \underline i}(\underline x): \underline i \in B_{\underline l} \right\} 
\end{eqnarray*}
A representation of the space $W_{\underline l}^{(L)}$ is given in dimension 1 on figure \ref{figWFunction1D}. 
\begin{figure}[h]
\centering
\includegraphics[width=5cm]{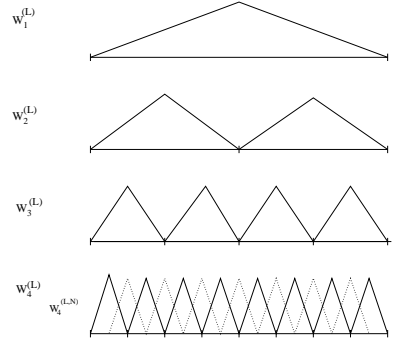}
\caption{One dimensional $W^{(L)}$ spaces : $W_1^{L)}$, $W_2^{(L)}$, $W_3^{(L)}$, $W_4^{(L)}$ and the nodal representation $W^{(L,N)}_4$ }
\label{figWFunction1D}
\end{figure}
The sparse grid space is defined as :
\begin{eqnarray*}
V_n &=& \underset{|\underline l|_1 \le n+d-1}{\oplus} W_{\underline l}
\label{classSparse}
\end{eqnarray*}
\begin{remark}
The conventional full grid space is defined as  $V_n^F =  \underset{|\underline l|_\infty \le n}{\oplus} W_{\underline l}$
\end{remark} 
At a space of hierarchical increments $W_{\underline l}^{(L)}$ corresponds a space of nodal function $W_{\underline l}^{(L,N)}$ such that
\begin{eqnarray*}
 W_{\underline l}^{(L,N)} &:=& span\left \{ \phi^{(L)}_{\underline l, \underline i}(\underline x): \underline i \in B^N_{\underline l} \right\}
\end{eqnarray*}
with  $$B^N_{\underline l} := \left \{ \underline i: 1 \le i_j \le 2^{l_j}-1,   1 \le j \le d \right\}.$$
On figure \ref{figWFunction1D} the one dimensional nodal base $W_4^{(L,N)}$ is spawned by $W_4^{(L)}$ and the dotted basis function.
The space $V_n$ can be represented as the space spawn by the $W^{(L,N)}_{\underline l}$ such that $|\underline l|_1= n+d-1$:
\begin{eqnarray}
V_n &=&  span\left  \{ \phi^{(L)}_{\underline l, \underline i}(\underline x): \underline i \in B_{\underline l}^N ,  |\underline l|_1 = n+d-1 \right \}
\label{classNodl}
\end{eqnarray}
A function $f$ is interpolated on the hierarchical basis as 
\begin{eqnarray}
I^{(L)}(f) &=& \sum_{|\underline l|_1 \le n+d-1,  \underline i \in B_{\underline l}} \alpha_{\underline l, \underline i}^{(L)} \phi^{(L)}_{\underline l, \underline i}
\label{basisRecons}
\end{eqnarray}
where the coefficient of the hierarchical basis $\alpha_{\underline l, \underline i}^{(L)}$ are called the surplus (we give on figure \ref{figWSurplus1D} a representation of these coefficients).
\begin{figure}[h]
\centering
\includegraphics[width=6cm]{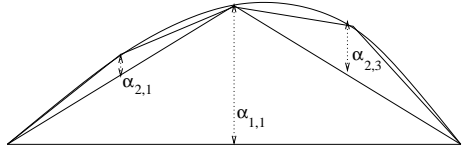}
\caption{Example of hierarchical coefficients}
\label{figWSurplus1D}
\end{figure}
These surplus associated to a function $f$ are calculated  in the one dimension case for a node $m = x_{l,i}$  as the difference of the value of the function at the node and the linear representation of the function calculated with neighboring nodes. For example on figure  \ref{figFatherTree}, the hierarchical value is given by the relation :
\begin{eqnarray*}
\alpha^{(L)}(m) := \alpha_{l, i}^{(L)} = f(m) -0.5 (f(e(m)) + f(w(m)))
\end{eqnarray*}
where $e(m)$ is the east neighbor of $m$ and $w(m)$ the west one.
The procedure is generalized in $d$ dimension by successive hierarchization in all the directions.
\begin{figure}[h]
\centering
\includegraphics[width=6cm]{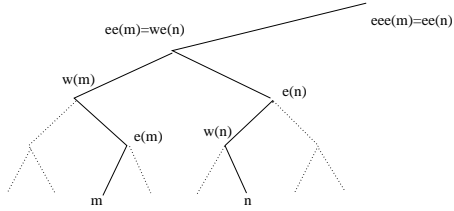}
\caption{Node involved in linear, quadratic and cubic representation of a function at node $m$ and $n$}
\label{figFatherTree}
\end{figure}
On figure \ref{figureSubSpace2D}, we give a representation of the $W$ subspace for $\underline l \le 3$ in dimension 2.\\
\begin{figure}[h]
\centering
\includegraphics[width=6cm]{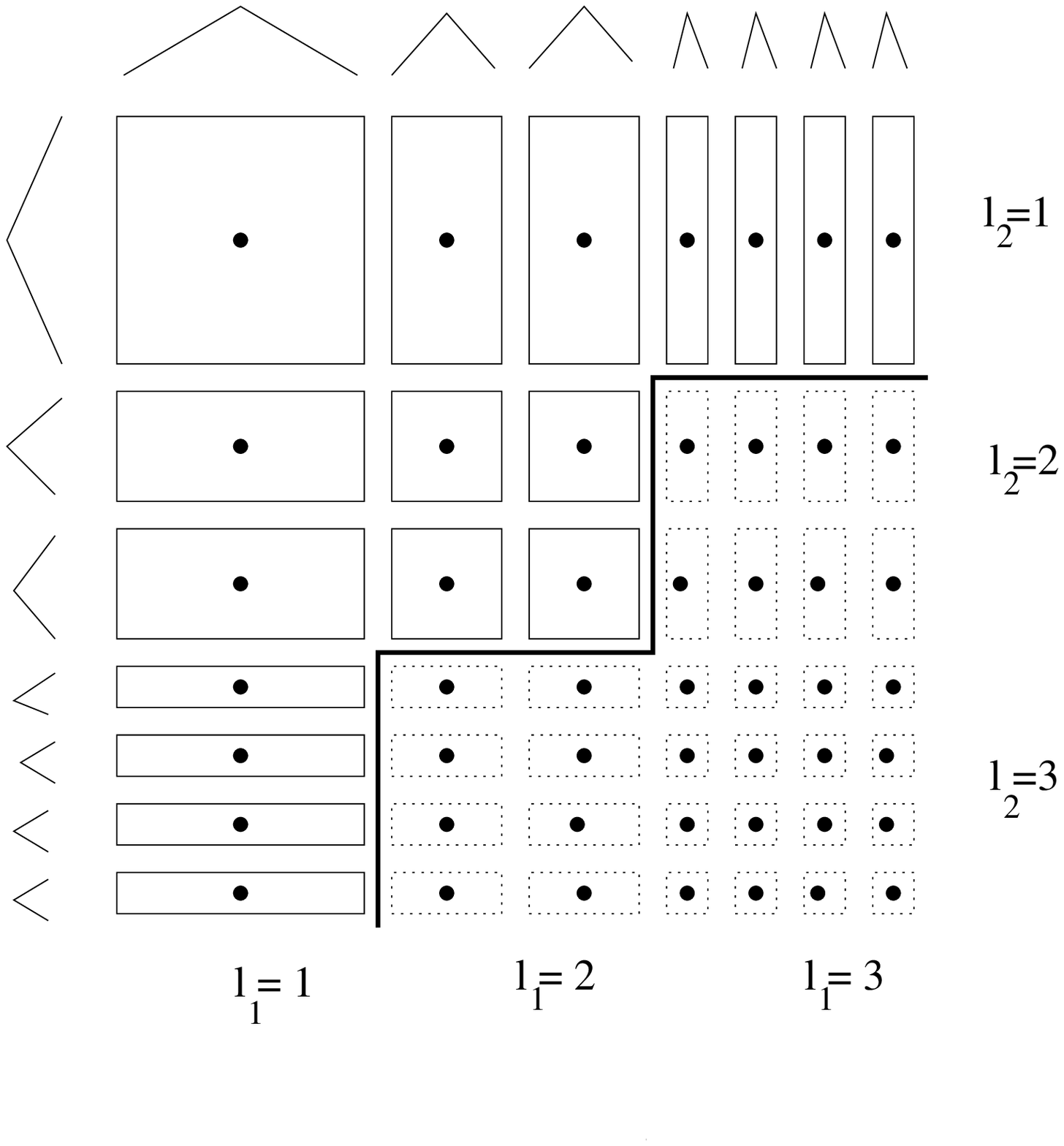}
\caption{The two dimensional subspace $W_{\underline l}^{(L)}$ up to $l=3$ in each dimension. The supplemental hierarchical functions corresponding to  an approximation
on the full grid are given in dashed lines.}
\label{figureSubSpace2D}
\end{figure}
In order to deal with functions not null at the boundary, two more basis are added to the first level as shown on figure \ref{figWFunction1DBound}.
\begin{figure}[h]
\centering
\includegraphics[width=5cm]{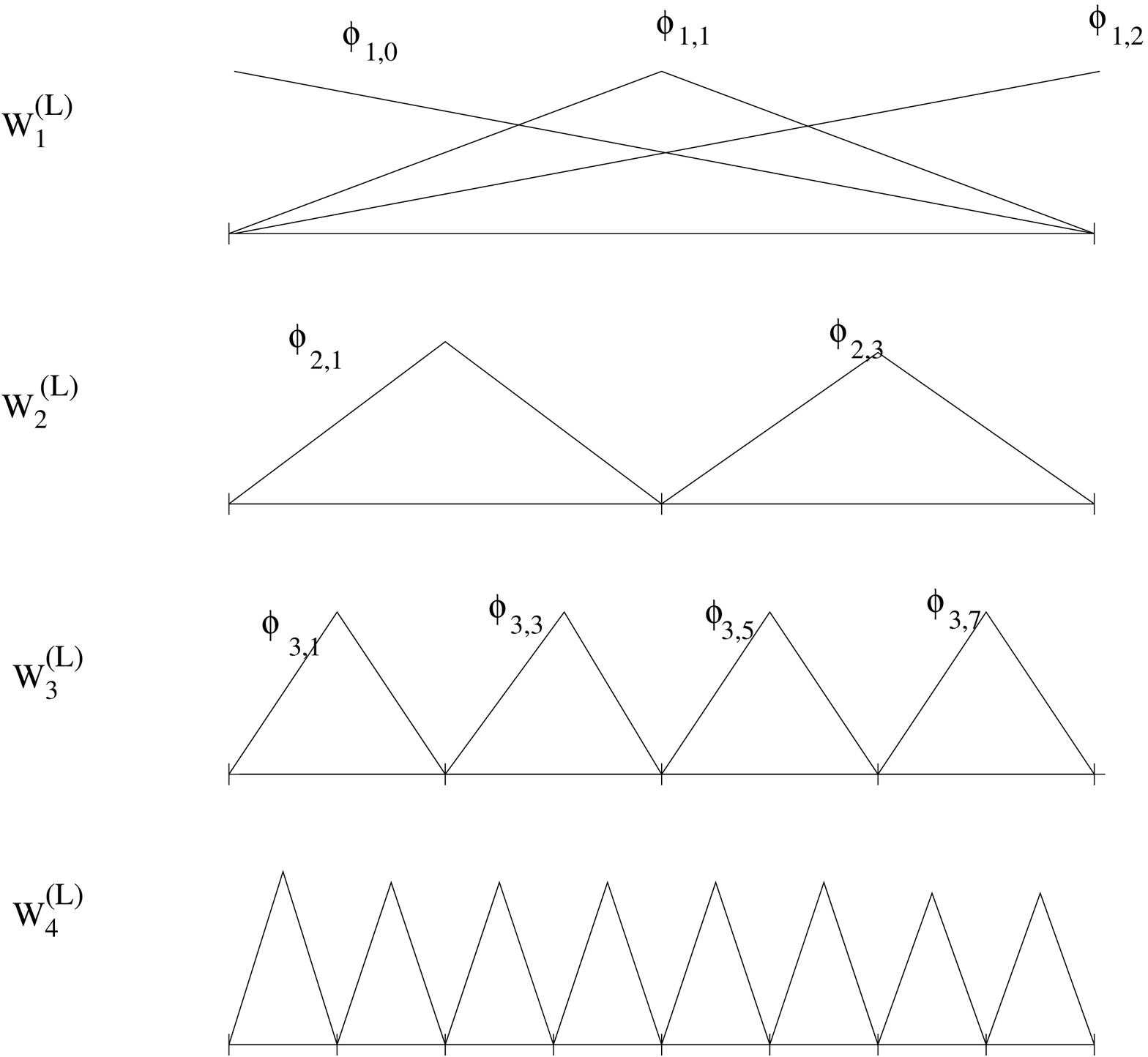}
\includegraphics[width=5cm]{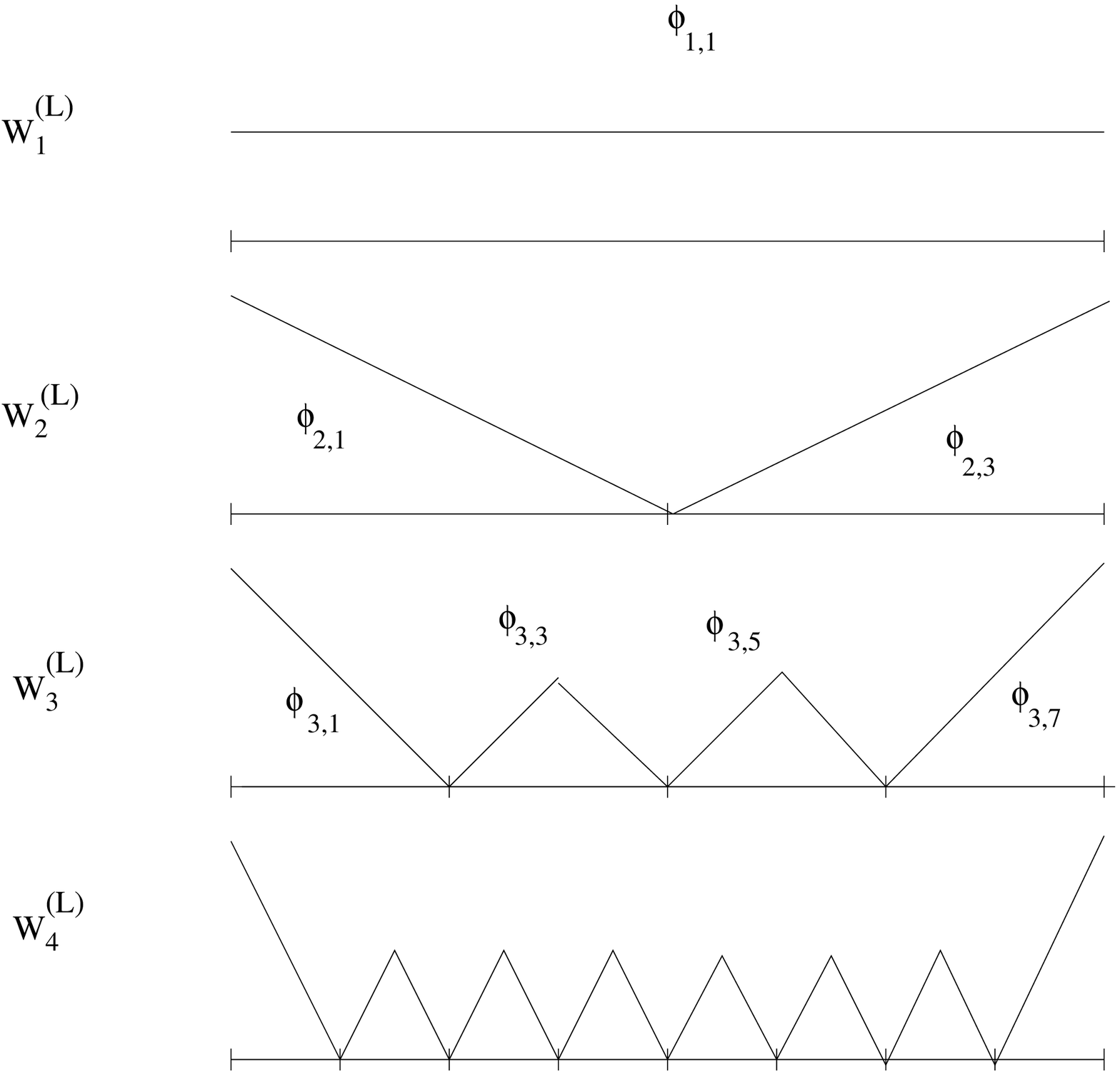}
\caption{One dimensional $W^{(L)}$ spaces  with linear functions with ``exact '' boundary (left) and ``modified '' boundary (right) : $W_1^{(L)}$, $W_2^{(L)}$, $W_3^{(L)}$, $W_4^{(L)}$ }
\label{figWFunction1DBound}
\end{figure}
This approach results in many more points than the one without the boundary. As noted in \cite{pfluger} for n =5, in dimension 8 you have nearly 2.8 millions points in this approximation but only 6401 inside the domain.  If the boundary conditions are not important (infinite domain truncated in finance for example) the hat functions near the boundaries are modified by extrapolation (see figure \ref{figWFunction1DBound}) as explained in \cite{pfluger}.
On level 1, we only have one degree of freedom assuming the function is constant on the domain. On all other levels, we extrapolate linearly towards the boundary the left and right basis functions, other functions remaining unchanged. So the new functions basis in 1D $\tilde \phi$ becomes
\begin{equation}
\begin{array}{ccc}
\tilde \phi_{l,i}^{(L)}(x) &=& \left \{ \begin{array}{cc}
                                  1   &   \mbox{ if } l=1 \mbox{ and } i =1 \\
                                   \left \{ \begin{array}{cc}
                                         2 -2^l x &  \mbox{ if } x  \in [0,2^{-l+1}] \\
                                         0        & \mbox{ else}
                                         \end{array} \right \}  & \mbox{ if } l > 1 \mbox{ and } i=1 \\
                                   \left \{ \begin{array}{cc}
                                         2^l (x-1) +2 &  \mbox{ if } x  \in [1-2^{-l+1},1] \\
                                         0        & \mbox{ else }
                                         \end{array} \right \}  & \mbox{ if } l > 1 \mbox{ and } i= 2^l-1 \\
                                   \phi_{l,i}^{(L)}(x) & \mbox{ otherwise } 
                                   \end{array}
                               \right. \nonumber
\end{array}
\end{equation}
%\begin{remark}\\
%It is clear that this approach doesn't see the boundary condition evolving with time. So only problems with
%\end{remark}
The interpolation error associated to the linear operator $I^1 := I^{(L)}$ is linked to the regularity of the cross derivatives of the function \cite{bungartz2,bungartz3,bungartz4}.
If $f$ is null at the boundary and  admits derivatives such that $ || \frac{\partial^{2d} u}{\partial x_1^{2} ... \partial x_d^{2}}||_{\infty}  < \infty$ then
\begin{eqnarray}
|| f - I^1(f)||_\infty =  O(N^{-2} log(N)^{d-1})
\label{ErrorInterpLin}
\end{eqnarray}
\subsection{High order case}
\label{highOrderSection}
As explained in \cite{warin2}, the observed convergence of the Semi Lagrangian method is slow.  Changing the interpolator permits to get a higher rate of convergence mainly in region where the solution is smooth.
Following \cite{bungartz3} and \cite{bungartz4}, it is possible to get higher interpolators. 
Using a quadratic interpolator, the reconstruction on the nodal basis gives a quadratic function
on the support of the previously defined hat function and a continuous function of the whole domain.
The polynomial quadratic basis is defined on $[2^{-l}(i-1),2^{-l}(i+1)]$ by $$\phi_{ l, i}^{(Q)}(x) = \phi^{(Q)}(2^l x-i)$$ with
$\phi^{(Q)}(x)= 1- x^2$. \\
The hierarchical surplus (coefficient on the basis) in one dimension is the difference between the value function at the node and the quadratic representation of the function using nodes available at the preceding level. With the notation of figure  \ref{figFatherTree} 
\begin{eqnarray*}
\alpha(m)^{(Q)} & =&  f(m) -(\frac{3}{8} f(w(m))+\frac{3}{4} f(e(m)) - \frac{1}{8} f(ee(m))) \\
               &= & \alpha(m)^{(L)}(m) -\frac{1}{4}\alpha(m)^{(L)}(e(m))\\
               &= & \alpha(m)^{(L)}(m) -\frac{1}{4}\alpha(m)^{(L)}(df(m))
\end{eqnarray*}
where $df(m)$ is the direct father of the node $m$ in the tree.\\
Once again the quadratic surplus in dimension $d$ is obtained by successive hierarchization in the different dimensions.\\
In order to take into account the boundary conditions, two linear functions $1-x$ and $x$ are added at the first level  (see figure \ref{figWFunction1DQuadBound}).\\
A version with modified boundary conditions can be derived for example by using linear interpolation at the boundary such that 
\begin{equation}
\begin{array}{ccc}
\tilde \phi_{l,i}^{(Q)}(x) &=& \left \{ \begin{array}{cc}
                                  \tilde \phi_{l,i}^{(L)}   &   \mbox{ if } i=1 \mbox{ or } i =  2^l-1 , \\
                                  \phi_{l,i}^{(Q)}(x) & \mbox{ otherwise } 
                                   \end{array}
                               \right. \nonumber
\end{array}
\end{equation}
\begin{figure}[h]
\centering
\includegraphics[width=5cm]{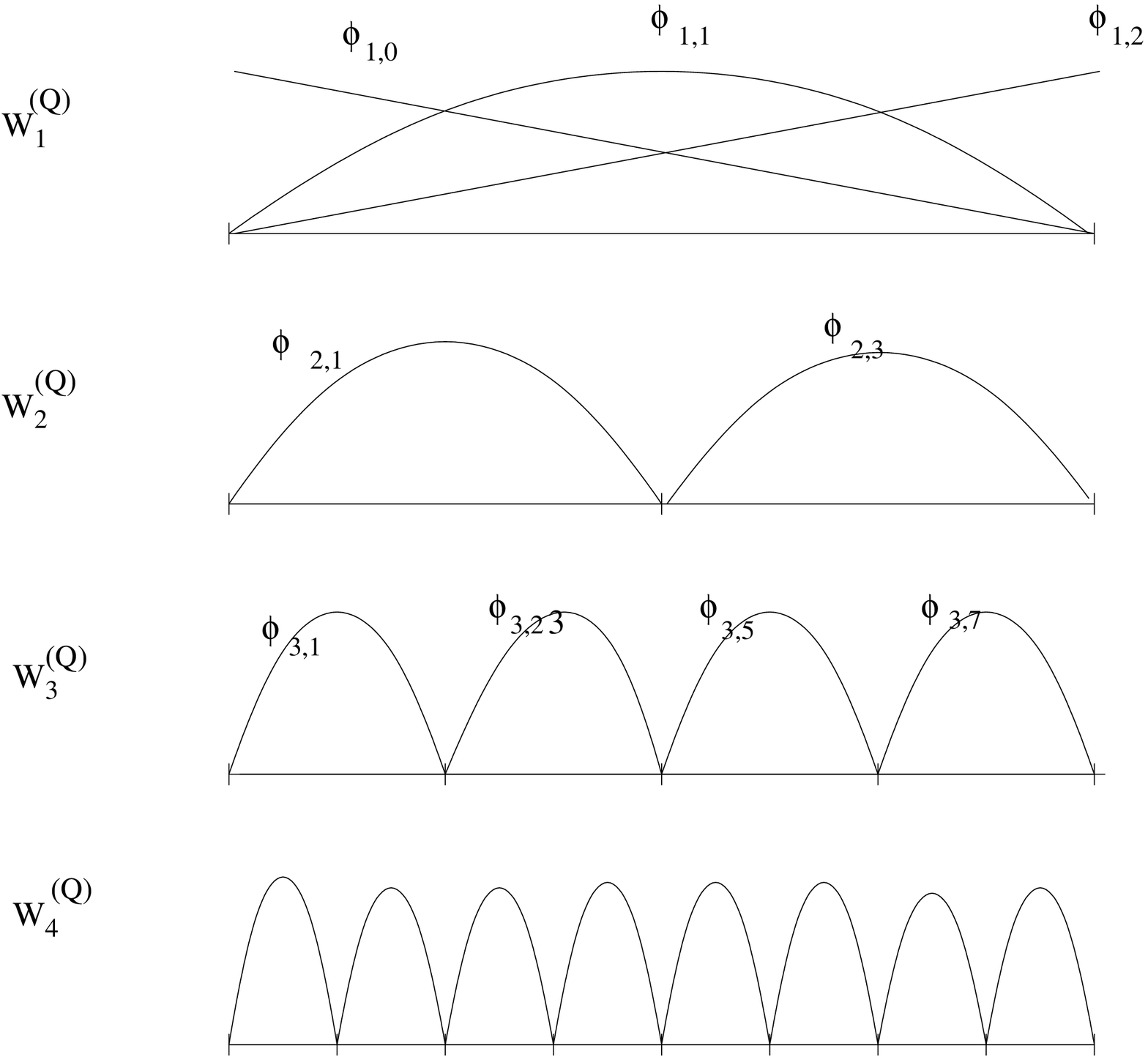}
\includegraphics[width=5cm]{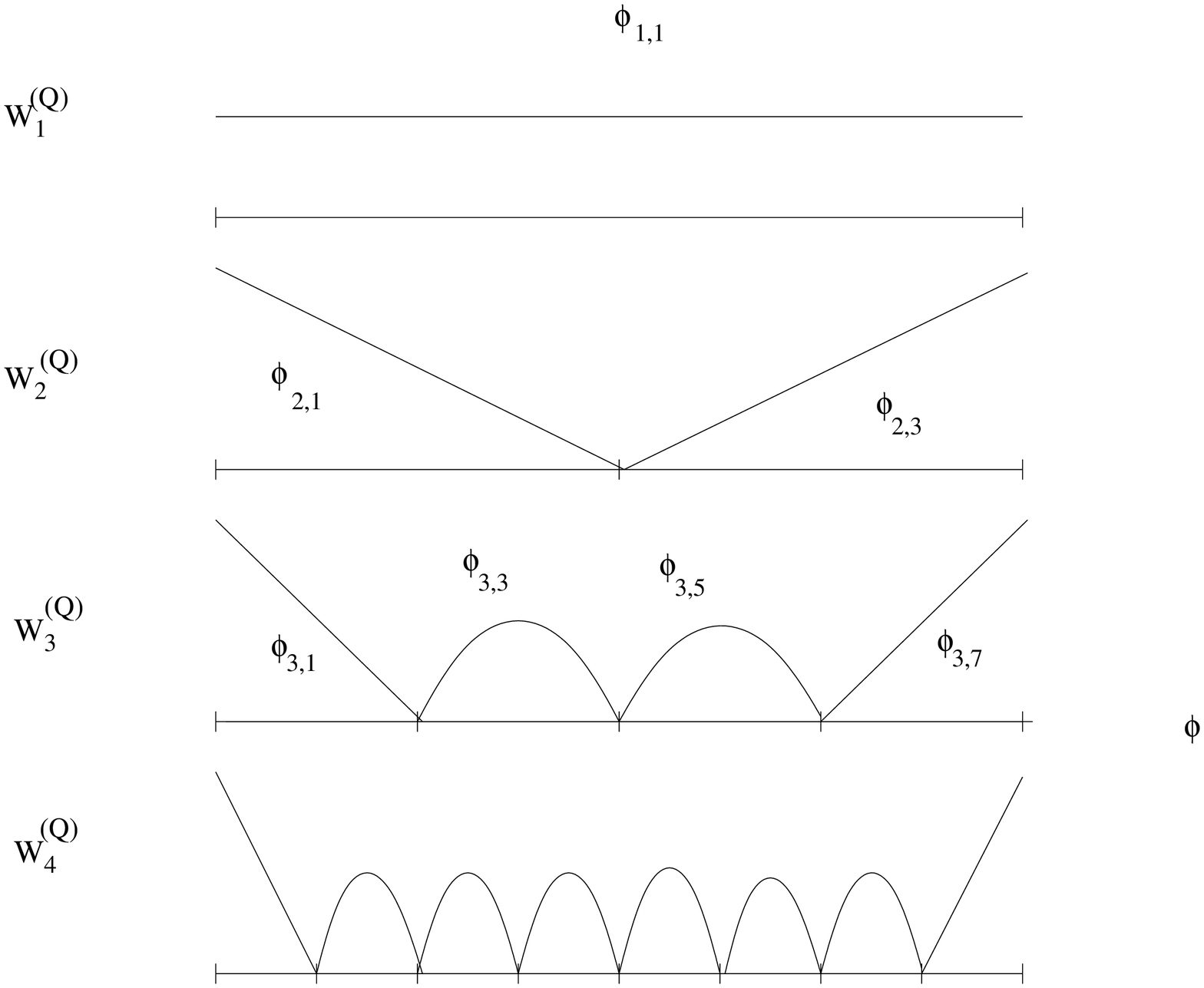}
\caption{One dimensional $W^{(Q)}$ spaces  with quadratic with ``exact'' boundary (left) and ``modified'' boundary (right) : $W_1^{(Q)}$, $W_2^{(Q)}$, $W_3^{(Q)}$, $W_4^{(Q)}$ }
\label{figWFunction1DQuadBound}
\end{figure}
In the case of the cubic representation, on  figure \ref{figFatherTree} we need 4 points to define a function basis. In order to keep the same data structure, we use a cubic function basis  at node $m$ with value 1 at this node and 0 at the node $e(m)$, $w(m)$ and $ee(m)$ and we only keep the basis function between $w(m)$ and $e(m)$ \cite{bungartz3}. \\
Notice that there are two kinds of basis function depending of the position in the tree. The basis functions are given on  $[2^{-l+1}i,2^{-l+1}(i+1)]$ by
\begin{eqnarray*}
\phi_{ l, 2i+1}^{(C)}(x) & = & \phi^{(C),1}(2^l x-(2i+1)) , \mbox{ if } i \mbox{ even }  \\
 & = & \phi^{(C),2}(2^l x-(2i+1)) , \mbox{ if } i  \mbox{ odd } 
\end{eqnarray*}
with $\phi^{(C),1}(x) = \frac{(x^2-1)(x-3)}{3}$, $\phi^{(C),2}(x) = \frac{(1-x^2)(x+3)}{3}$.\\
The coefficient surplus can be defined as before as the difference between the value function  at the node and the cubic representation of the function at the father node. Because of the two basis functions involved there are two kind of cubic coefficient.
\begin{itemize}
\item 
For a node $m = x_{l,8i+1}$ or   $m = x_{l,8i+7}$ , $\alpha^{(C)}(m)  = \alpha^{(C,1)}(m)$, with $$\alpha^{(C,1)}(m) =   \alpha^{(Q)}(m) - \frac{1}{8} \alpha^{(Q)}(df(m))$$ 
\item 
For a node $m = x_{l,8i+3}$ or   $m = x_{l,8i+5}$ , $\alpha^{(C)}(m)  = \alpha^{(C,2)}(m)$, with $$\alpha^{(C,2)}(m) =    \alpha^{(Q)}(m) + \frac{1}{8} \alpha^{(Q)}(df(m))$$
\end{itemize} 
Notice that a cubic representation is not available for $l=1$ so a quadratic approximation is used. 
As before boundary conditions are treated by adding two linear functions basis at the first level and a modified version is available. We choose the following basis functions as defined on figure \ref{figWFunction1DCubicBound} :
\begin{equation}
\begin{array}{ccc}
\tilde \phi_{l,i}^{(C)}(x) &=& \left \{ \begin{array}{cc}
                                  \tilde \phi_{l,i}^{(Q)}   &   \mbox{ if } i \in \{ 1, 3, 2^{l}-3, 2^l-1 \} , \\
                                  \phi_{l,i}^{(C)}(x) & \mbox{ otherwise } 
                                   \end{array}
                               \right. \nonumber
\end{array}
\end{equation}
\begin{figure}[h]
\centering
\includegraphics[width=5cm]{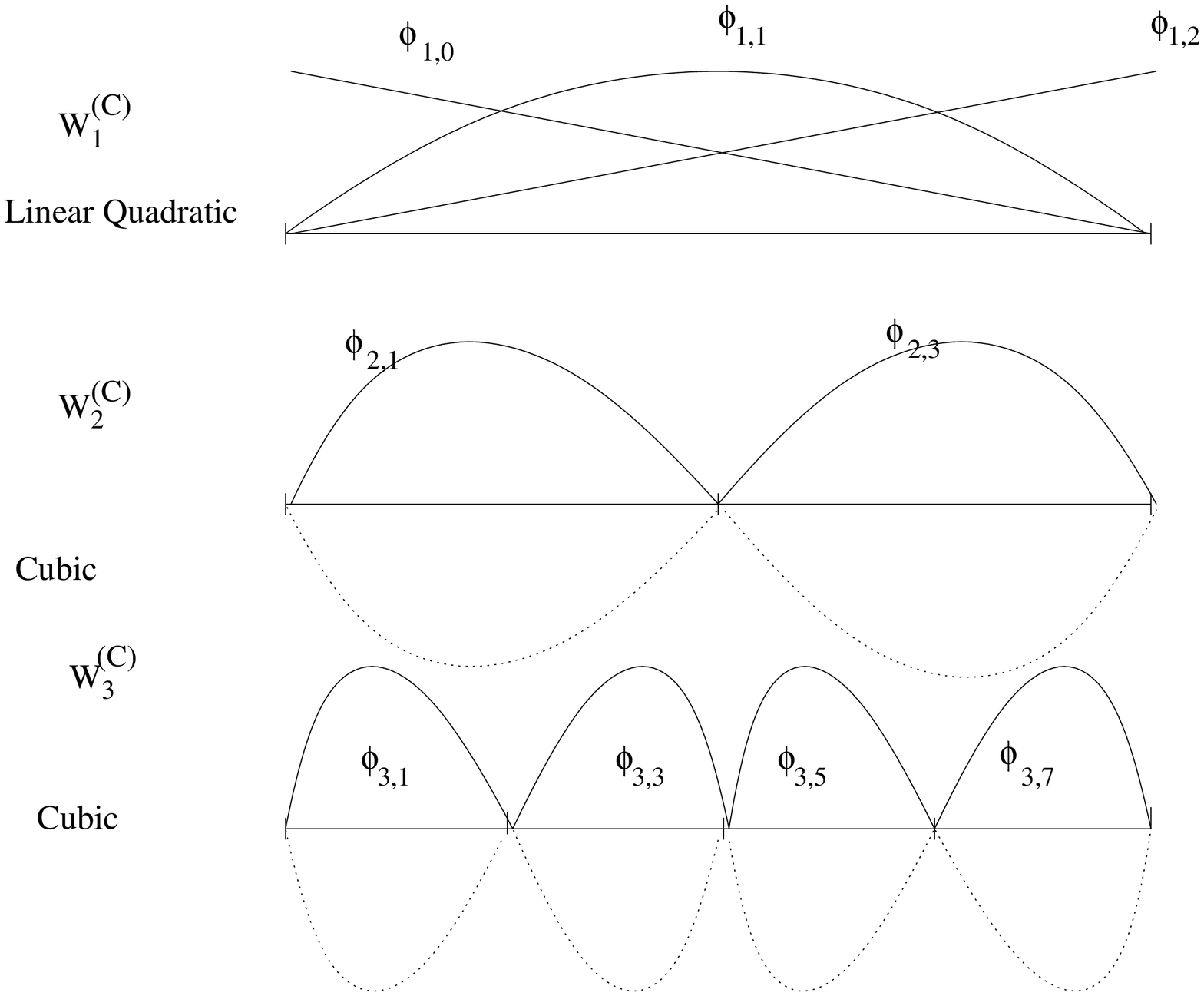}
\includegraphics[width=5cm]{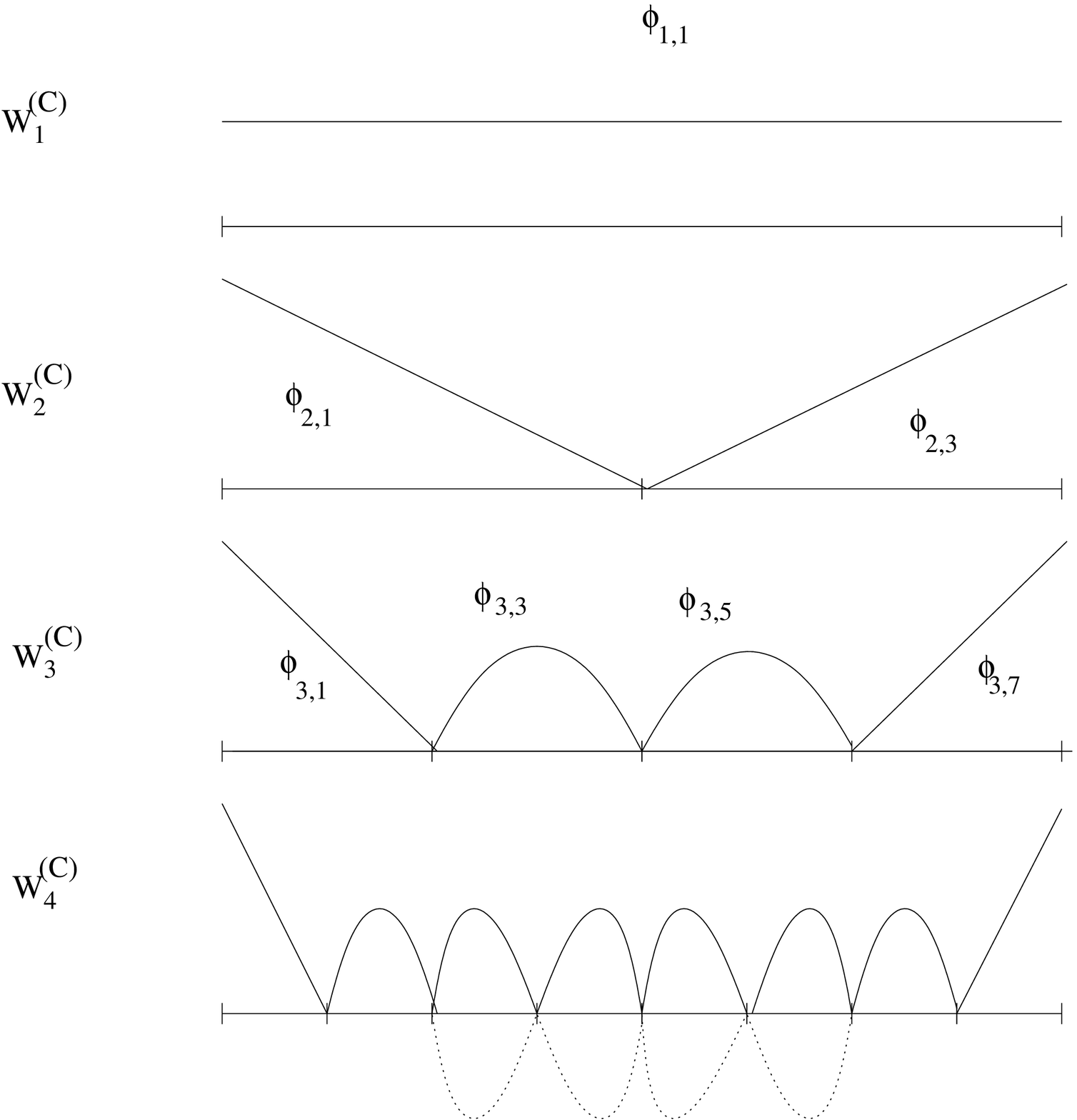}
\caption{One dimensional $W^{(C)}$ spaces  with cubic and ``exact`` boundary (left) and ``modified'' boundary (right) : $W_1^{(C)}$, $W_2^{(C)}$, $W_3^{(C)}$, $W_4^{(C)}$ }
\label{figWFunction1DCubicBound}
\end{figure}

According to  \cite{bungartz2,bungartz3,bungartz4}, if the function $f$  is null at the boundary and admits derivatives such that 
$\sup_{\alpha_i \in \{2,..,p+1\}} \left \{ || \frac{\partial^{\alpha_1+..+\alpha_d} u}{\partial x_1^{\alpha_1} ... \partial x_d^{\alpha_d}}||_{\infty} \right \} < \infty$
then the interpolation error can be generalized for $I^2:= I^{(Q)}$, $I^3 :=  I^{(C)}$ by :
\begin{eqnarray}
|| f - I^p(f)||_\infty =  O(N^{-(p+1)} log(N)^{d-1}), \quad p=2,3
\label{ErrorInterp}
\end{eqnarray}

\section{Truncated interpolator}
\label{truncationSection}
In order to get a higher rate of convergence where the solution is regular, we'd like to use a high order interpolator.
However, the use of a high order interpolator doesn't permit to prove convergence of the calculated solution to the viscosity solution of our problem.
In order to recover this convergence, we will use a truncation similar to the one developed in  \cite{warin1}.
This approach has to be adapted to the sparse grid case.
For high order interpolator (quadratic and cubic) we truncate  the interpolated value at a point $x$ as follows :
let's define for the nodal basis functions $\phi$ :
\begin{eqnarray}
K^n(x) & =& \{(\underline l, \underline i), \mbox{ such that } |\underline l|_1= n+d-1, \underline{i} \in B^N_{\underline l}  , \mbox{ and }  x \in 
\supp \phi_{\underline l, \underline i} \}
%\mbox{support} \phi_{\underline l, \underline i} \}
%\label{Kdef}
\nonumber
\end{eqnarray}
\begin{remark}
By construction, the nodal basis function $\phi^{(L)}$, $\phi^{(Q)}$, $\phi^{C)}$ have the same support.
\end{remark}
Defining the maximum and minimum values taken by a function $f$ at theses nodes
\begin{eqnarray}
\underline{f}(x) & = & \min( f(x_{\underline l, \underline i})/ (\underline l, \underline i) \in K^n(x)) \nonumber \\
\bar{f}(x)      &  = & \max( f(x_{\underline l, \underline i})/ (\underline l, \underline i) \in  K^n(x))  \label{truncation}
\end{eqnarray}
And truncate
\begin{eqnarray*}
I^{p,c}(f)(x) & = &  \underline{f}(x) \vee I^p(f)(x) \wedge \bar{f}(x) 
\end{eqnarray*}
\begin{remark}
We  have of course  $I^{1,c}(f)=I^{1}(f)$ but in the sequel we keep the notation $I^{1,c}$ for genericity.
\end{remark}
When the truncation is really achieved, the rate of the interpolation error \reff{ErrorInterp} cannot be better than the
one obtained by a linear interpolation given by equation \reff{ErrorInterpLin}.
We hope that the truncation will only be really achieved at points where the solution is not  regular.

\section{Spatially adaptive sparse grids}
When the solution is not smooth, typically Lipschitz, there is no hope to get convergence results  for classical Sparse Grids (see above the interpolation error linked to the cross derivatives of the function).
So classical sparse grids have to be adapted such that the solution is refined at points of irregularity. 
In all adaptations methods hierarchical surplus  $\alpha_{\underline l, \underline i}$ are used to get  an estimation of the local error. These coefficients give an estimation of the smoothness of the function value at the discrete points  by representing the discrete mix second derivative of the function.
There is mainly two kinds of adaptation used :
\begin{itemize}
\item the first one is performing local adaptation and only adds points locally \cite{griebel,bungartz1,griebel1,ma},
\item the second one  is performing adaptation at the level of the hierarchical space $W_{\underline l}$ (anisotropic sparse grid).
This approach detects important dimensions  that needs refinement and refines all the points in this dimension \cite{gerstner}. This refinement is also achieved in areas where the solution can be smooth. A more local version has been developed in \cite{roberts}.
\end{itemize}
At the first date, the dimension adaptation developed in \cite{gerstner} is used  to get a good approximation of the function.
Because we want to implement local adaptation with refinement and coarsening, the first kind of adaptation is then used at each time step.\\
\begin{algorithm}
\caption{Algorithm for adaptation}
\label{algoAdapt}
\begin{algorithmic}
\STATE Use dimension adaptation at the first time step 
\STATE Mesh coarsening to initialize the refinement for next step
\FOR{ each time step} 
    \WHILE{precision and maximal level not reached}
    \STATE  Refine 
    \ENDWHILE
    \STATE Store the solution with the refined grid
    \STATE Mesh coarsening  for next time step
\ENDFOR
\end{algorithmic}
\end{algorithm}
The algorithm used during the resolution is given in \ref{algoAdapt}.
Refinement is achieved on
points that have the maximal surplus above the precision required and that are located at the leave of the multidimensional tree. Each point  has $2^d$ sons. As pointed out in \cite{bokanowski} some points constructed may not have ancestors in some dimension. In order to avoid holes in the structure missing fathers are added.\\
At each time step, the adaptation is achieved iteratively if a maximal level of refinement is not reached.
In order to prepare the next time step, the grid is coarsened by deleting points that correspond to leaves in the multidimensional tree if the associated surplus are less then 10 times the precision required.
A recent description of the algorithm of refinement and mesh coarsening can be found  in \cite{bokanowski}.

 \section{Discretized scheme and convergence analysis}
\label{analysisSection}
We take the following noting notations :  
\begin{itemize}
\item  $P_S^N(t)$ is the set of all the points $(\underline l, \underline i)$ , $\underline i \in  B_{\underline l}$, at date $t$ of the adapted Sparse Grid meshing such that in each direction the maximal level is $N$  ($|\underline l|_\infty \le N$)
\item $P_F^N$ is the set of all points $(\underline l, \underline i)$ belonging to the corresponding full grid :
\begin{eqnarray*}
P_F^N & = & \{(\underline l, \underline i) / | \underline l|_\infty \le N, \underline i \in B_{\underline l} \}
\end{eqnarray*}
\end{itemize}
We note  $I^{p,c,Full}_{N} $ the full grid operator on $P_F^N$ such that  
$$I^{p,c,Full}_{N}(f)=  I^{p,c}[ (f(x_{\underline l, \underline i}))_{(\underline l, \underline i) \in P_F^N}],$$
  $p=1$ to $3$ (being linear, quadratic or cubic). With similar notation
$$I^{p,c,S}_{N}(f)=  I^{p,c}[ (f(x_{\underline l, \underline i}))_{(\underline l, \underline i) \in P_S^N}].$$
As we increase the number of points during adaptation while keeping the level below $N$ , the interpolator $I^{p,c,S}_N$  converges towards the 
full grid interpolator $I^{p,c,Full}_{N}$. 
 Using the same scheme as the one defined in \cite{warin1},  for each point of the adapted meshing $x_{\underline l, \underline i} \in P_S^N(t+h)$,
$v_{\underline l, \underline i}(t+h)$  the estimation of the value function at the date $t+h$ and the point $x_{\underline l, \underline i}$ is calculated by the scheme :
\begin{eqnarray}
v_{\underline l, \underline i}(t+h) & =&  v_{\underline l ,\underline i}(t)  \nonumber  \\
& & + \inf_{a \in \mathop{A}}  \left [  ( L_{a,h} I^{p,c,S}_N[(v_{\underline{ll} ,\underline{ii}}(t))_{(\underline{ll} ,\underline{ii}) \in P_S^N(t)}](x_{\underline l ,\underline i})  \right]  \label{hjbInterpGeneral} 
\end{eqnarray}
With this algorithm the solution can be calculated recursively for $t_j= j h$, $j=1, n$. 
An estimation of the solution at date $t$ is given
\begin{eqnarray*}
\tilde v(t,x) & =& I^{p,c,S}_N[ (v_{\underline l, \underline i}(t))_{(\underline l, \underline i) \in  P_S^N(t)} ](x)
\end{eqnarray*}

\subsection{Convergence rate for the linear interpolator}
As we increase the number of points during adaptation while keeping the level below $N$ , the interpolator with no truncation $I^{1,S}_N$  converges towards the 
full grid interpolator $I^{1,Full}_{N}$.
We give some assumptions for the convergence of the adaptation
\begin{assumption}
\label{assump1}
We suppose that  all surplus missing associated to points to obtain the full grid operator are below $\epsilon$ and that $M$ is the number of missing points
\end{assumption}
We get the following result
\begin{lemma}
Under assumption \ref{assump1}, for every Lipschitz function $f$ 
\begin{eqnarray*}
||I^{1,S}_N(f)-f||_\infty  & \le & K 2 ^{-N} + M \epsilon
\end{eqnarray*}
where $M$ is the number of missing points
\end{lemma}
\begin{proof}
 As in \cite{ma} :
\begin{eqnarray*}
||I^{1,S}_N(f)-f||_\infty  & \le  & ||I^{1,S}_N(f) -I^{1,Full}_{N}(f)||_\infty + || I^{1,Full}_{N}(f) -f||_\infty
\end{eqnarray*}
$f$ function being Lipschitz, we get  the classical result (see \cite{warin1}) $|| I^{1,Full}_{N}(f) -f|| \le  K 2 ^{-N}$.
On the other hand by assumption \ref{assump1}
\begin{eqnarray*}
||I^{1,S}_N(f) -I^{1,Full}_{N}(f)||_\infty & \le & \sum_{|\underline l|_{\infty} \le N,  x_{\underline l, \underline i} \notin P_S^N  }    |\alpha^{L}(f)_{\underline l, \underline i}| ||\phi^{(L)}_{\underline l, \underline i}||_\infty \\
& \le & M \epsilon
\end{eqnarray*}
because the function basis is bounded by one.
\end{proof}
\begin{remark}
The adaptation algorithm is stopped when all the surplus calculated are below a given level $\epsilon$. It doesn't prove that while going on
the refinement, we only get surplus below this threshold.
\end{remark}
The next lemma proves that if missing surplus to get the full grid approximation at each time step are bounded by a given value and if the number of missing points is bounded uniformly in time, then
the Sparse Grid Semi-Lagrangian approximation converges towards the viscosity solution of the problem for the linear interpolator.
\begin{assumption}
\label{assump3} 
Assumption \ref{assump1} is satisfied with a number $M$ independent on the  time step number . 
\end{assumption}
\begin{theorem}
Under assumption \ref{assump3}, using the linear interpolator, the scheme \reff{hjbInterpGeneral} satisfies for all $j \in [0,T/n]$
\begin{eqnarray}
|| \tilde v(t_j,.)- v(t,.)||_\infty & \le  & C (h^{\frac{1}{4}} + 2^{-N} + M \frac{\epsilon}{h})
\end{eqnarray}
\label{convRateL}
\end{theorem}
\begin{proof}
 We choose $h \le   1$ and satisfying the hypothesis of proposition \reff{converDisH}.
Recalling that $v_h$ is the solution of equations \reff{hjbCamilli} and  \reff{hjbCamilli1},
 we directly estimate  $\tilde v - v_h$ . Introducing
\begin{eqnarray*}
e(t) & =& ||\tilde v(t,.) - v_h(t,.)||_{\infty} \nonumber
\end{eqnarray*}
we estimate
\begin{eqnarray}
|\tilde v(t,x)- v_h(t,x)|  & \le &  |I^1[ (v_{\underline l, \underline i}(t))_{(\underline l, \underline i) \in  P_S^N(t)}](x) -I^1[ (v_{\underline l, \underline i}(t))_{(\underline l, \underline i) \in  P_F^N}](x)| + \nonumber\\
& & 
    |I^1[ (v_{\underline l, \underline i}(t))_{(\underline l, \underline i) \in  P_F^N}- v_h(x,t)|  \nonumber \\
& \le & M \epsilon + \sup_{(\underline l ,\underline i) \in  K^{N,FULL}(t,x)}  |v_{\underline l, \underline i}(t) - v_h(t,x)|
\label{eqnJac}
\end{eqnarray}
where $K^{N,FULL}(t,x)$ defines the set  of all the points $(\underline l ,\underline i)$ of $P_F^N$ corresponding to   edges of the nodal cell where $x$ lies in the full grid. Notice that
\begin{eqnarray}
 \forall (\underline l, \underline i) \in K^{n,FULL}(x), | x_{\underline l, \underline i} - x|_{\infty} < C 2^{-N}.
\label{distPoint}
\end{eqnarray}

For $(\underline l, \underline i) \in K^{N,FULL}(t,x)$, we introduce  $V :=  v_{\underline l, \underline i}(t) - v_h(t,x) $. \\
Using equation \reff{hjbInterpGeneral}, the classical relation $| inf . - inf . | \le sup | . - .|$, the fact that the data in equation \reff{hjb} belong to  $C_1(Y)$ and equation \reff{distPoint} we get 
\begin{eqnarray}
|V|  & \le &     \frac{1}{2q}  \sum_{i=1}^q   \left [\sup_a | I^1[ (v_{\underline{ll}, \underline{ii}}(t-h))_{(\underline{ll}, \underline{ii}) \in P_S^N(t-h)}](\phi^{+}_{a,h,i}(t-h,x_{\underline l, \underline i}))  \right. \nonumber\\
  & & - v_{h}(t-h,\phi^{+}_{a,h,i}(t-h,x))|  \nonumber \\
  &     &   + \sup_a | I^1[ (v_{\underline{ll}, \underline{ii}}(t-h))_{(\underline{ll}, \underline{ii}) \in P_S^N(t-h)}]( \phi^{-}_{a,h,i}(t-h,x_{\underline l, \underline i})) \nonumber \\
& &  \left . -    v_{h}(t-h,\phi^{-}_{a,h,i}(t-h,x)) |  \right ]  \nonumber \\
  &   & +   h \sup_a |c_a|_0 | I^1[ (v_{\underline{ll}, \underline{ii}}(t-h))_{(\underline{ll}, \underline{ii}) \in P_S^N(t-h)}](x_{\underline l, \underline i}) - v_h(t-h,x_{\underline l, \underline i})|   \nonumber \\
  & &  + h |v_h|_0 |c_a|_1 C 2^{-N} +    h \sup_a| f_a|_1  C 2^{-N} 
\label{eqDem1}
\end{eqnarray}
Similarly, we obtain : 
\begin{eqnarray*}
| \phi^{-}_{a,h,i}(t,x_{\underline l, \underline i}) -\phi^{-}_{a,h,i}(t,x) | & \le & C2^{-N} (1 +  \sup_a |b_a|_1 h  + \sup_a |(\sigma_a)_i|_1 \sqrt{h q}) \\
                                                        & \le &  C 2^{-N} (1+C (\sqrt{h q} + h))   \\
                                                        & \le &  C 2^{-N} 
\end{eqnarray*}
and similar result is obtained  for $\phi^{+}$.\\
Using the fact that $|v_h|_1$ is bounded independently on  $h$, the previous estimation for $\phi^{+}$,  one gets the estimate
\begin{eqnarray}
|  & I^1[ & (v_{\underline{ll}, \underline{ii}}(t-h))_{(\underline{ll}, \underline{ii}) \in P_S^N(t-h)}] \left(\phi^{+}_{a,h,i}(t-h,x_{\underline l, \underline i})\right)  - v_{h}(t-h,\phi^{+}_{a,h,i}(t-h,x))| \le  \nonumber \\
& &   || \tilde v(t-h,.) -  v_h(t-h,.)||_\infty +
C |v_h|_1 2^{-N} 
\label{eqDem2}
\end{eqnarray}
Using the fact that $|f_a|_1$, $|c_a|_1$  are bounded independently of  $a$  as in \cite{warin1}, equations \reff{eqDem1} and \reff{eqDem2} give: 
\begin{eqnarray}
|V| & \le &  || \tilde v(t-h,.) - v_h(t-h,.)||_\infty (1+ h \hat K ) + C 2^{-N} 
\label{eqDem3}
\end{eqnarray}
where the constant  $C$ depends on  $\tilde K$, $|v_h|_1$, $\hat K$.\\
So using equations \reff{eqnJac} and \reff{eqDem3} 
\begin{eqnarray*}
e(t) & \le &  (1+ h \hat K)   e(t-h) +  C ( 2^{-N}  + M \epsilon)
\end{eqnarray*}
Using the fact that $e(0)=0$ and using discrete Gronwall Lemma
\begin{eqnarray*}
e(t_j) & \le &  C  e^{ \hat C T}  (\frac{ 2^{-N}}{ h} + \frac{M \epsilon}{h}),   \quad \quad  \forall j  < \frac{T}{n}
\end{eqnarray*}
Moreover by using  $||\tilde v(t,.) - v(t,.)||_\infty  \le | \tilde v(t,.) - v_{h}(t,.)|_\infty + |v_{h} -v|_0$ and the proposition \reff{converDisH} we get $ \forall j \in (0,T/n)$ :
\begin{eqnarray}
|| \tilde v(t_j,.)- v(t_j,.)||_\infty & \le  & C (h^{\frac{1}{4}} + \frac{ 2^{-N}}{h} + M \frac{\epsilon}{h})
\end{eqnarray}

\end{proof}

\subsection{High order interpolator estimator}
\label{HighEstiSection}
The case of high order is less obvious in the general case. We will impose that the discretized  grid is refined enough for 2 schemes with solutions $\tilde v^{-}$ and $\tilde v^+$  converging towards  the viscosity solution and such that $\tilde v^{-} \le v \le \tilde v^{+}$.\\
We introduce the interpolation operators on the sparse grid (see definition in equation \reff{truncation}):
\begin{eqnarray*}
I_{-,N}^{S}(f)(x) & = &  \underline{f}(x),
\end{eqnarray*}
\begin{eqnarray*}
I_{+,N}^S(f)(x) & = &  \bar{f}(x),
\end{eqnarray*}
and the equivalent on  the full grid $I_{-}^{FULL}$, $I_{+}^{FULL}$.
Let's define $\tilde v_{-}(t_j,.)$, $\tilde v_{+}(t_j,.)$ the value function defined by
\begin{eqnarray*}
\tilde v^{\pm}(t+h,x) & =& I_{\pm,N}^S[ (v^{\pm}_{\underline l, \underline i}(t+h))_{(\underline l, \underline i) \in  P_S^N(t+h)} ]
\end{eqnarray*}
with 
\begin{eqnarray}
v^{\pm}_{\underline l ,\underline i}(t+h,x)  & =&  v^{\pm}_{\underline l ,\underline i}(t)  + \inf_{a \in \mathop{A}}  \left [  ( L_{a,h} I^{\pm}[(v_{\underline{ll} ,\underline{ii}}(t))_{(\underline{ll} ,\underline{ii}) \in P_S^N(t)}](x_{\underline l ,\underline i})  \right]  \label{hjbInterpPlusMinus} 
\end{eqnarray}
We use the following assumption stating that in the previous schemes, the adaptation (used in our main scheme) is accurate enough.
\begin{assumption}
\label{assump4} 
We suppose that  all surplus missing  of the hierarchical decomposition of the  solution $v^{\pm}(t)$ on the sparse grid $P_S^N(t)$
associated to points missing  to obtain the full grid operator are below $\epsilon$  and that  $M$ the number of missing points is independent
on the time step.
\end{assumption}
\begin{assumption}
\label{assump5} 
The coefficient $c_a$ is positive. 
\end{assumption}
\begin{theorem}
Under assumptions \ref{assump4} and \ref{assump5}, using the interpolation operators $I^{p,c,S}_N$, $p=2, 3$, the scheme \reff{hjbInterpGeneral} satisfies for all $j \in [0,T/n]$
\begin{eqnarray}
|| \tilde v(t_j,.)- v(t,.)||_\infty & \le  & C (h^{\frac{1}{4}} + \frac{2^{-N}}{h} + M \frac{\epsilon}{h})
\end{eqnarray}
\label{convRateQC}
\end{theorem}
\begin{proof}
By recurrence, due to the truncation applied for quadratic or cubic operator and assumption \ref{assump5}   we easily get that
\begin{eqnarray}
\tilde v^{-}(t,x) & \le \tilde v \le \tilde v^{+}
\label{inBetween}
\end{eqnarray}
For $\tilde v^{-}(t,x)$ and  $\tilde v^{-}(t,x)$ we can use exactly the same procedure as in theorem \ref{convRateL}.
Because the solution is only Lipschitz, all interpolation errors are the same for constant per mesh interpolation as in the linear case.
Therefore we get :
\begin{eqnarray}
|| \tilde v^{-}(t_j,.)- v(t_j,.)||_\infty & \le  & C (h^{\frac{1}{4}} + \frac{2^{-N}}{ h} + M \frac{\epsilon}{h})
\label{infEst}
\end{eqnarray}
and
\begin{eqnarray}
|| \tilde v^{+}(t_j,.)- v(t_j,.)||_\infty & \le  & C (h^{\frac{1}{4}} + \frac{2^{-N}}{ h} + M \frac{\epsilon}{h})
\label{supEst}
\end{eqnarray}
The result is straightforward using equations \reff{inBetween}, \reff{infEst} and \reff{supEst}.
\end{proof}
\begin{remark}
Assumption \ref{assump5} is not a strict restriction : using a change of unknown, $v(t,x) = e^{-K t} u(t,x)$ with $K > -|c_a|_1$, $u$ satisfies assumption \ref{assump5}.
\end{remark}
\section{Numerical examples}
\label{NumSection}
Two data structures have been tested to store hierarchical coefficients : the most compact is the one developed in \cite{dirnstorfer} using a bitset to represent a position in the multidimensional tree. A second one tested is more classical : each node is represented by a multidimensional level and a multidimensional index. The efficiency of both data structure for interpolation is very similar so  we decide to give results with the second data structure. \\
A key point for the efficiency of the method is the reduction of the cost of the interpolation procedure. Due to the adaptation step, the structure is very irregular.
A classical algorithm to interpolate is to see the sparse grid as a recombination of full grids with different size mesh. The solution is interpolated on each full grid and all contributions are gathered to reconstruct the solution. This approach  in the multidimensional case is not effective due to adaptation.
An effective one consists in starting from the root node in the tree, then calculate its contribution to the solution and recursively
go down the tree in all the dimensions testing if the point belongs to the support of the left or the right node. In that way only basis functions whose support encompass the points are evaluated giving a contribution to the solution.\\
The different test cases can be divided into two kinds :
\begin{itemize}
\item in the first two cases, the problem needs to get a quite good approximation of the solution at the boundary. Therefore  we use the Sparse Grid method with the boundary points. No adaptation is used for these very simple cases.
\item in all the other cases, the boundary condition is not very important . Because we are interested in what is happening  far away from the boundary, one can only use an approximated solution on the boundary
or one can use the Sparse Grid using extrapolation to avoid boundary points.
The last cases are financial cases that permits to test the extrapolation method.
\end{itemize}
In all the cases given, the truncation appears to be useless, costing more than $5\%$  of the total consumption time but only modifying the solution calculated at the third digit. So the results are given without truncation.
\begin{remark}
Please notice that when the modified boundary is used no estimation of the error is available.
\end{remark}
\begin{remark}
When test cases are parallelized, the parallelization is only achieved by distributing the points where the value function must be estimated.
No parallelization of the hierarchization procedure and the adaptation phase (location of points to be refined, points to add to the structure, points to be removed when coarsening) is performed.
When a high number of processors is used, the time used in the non parallelized part of the software can be dominating. 
\end{remark}
\subsection{Two dimensional test cases}
\label{2DNumSection}
We take some test cases from \cite{warin1}
\subsubsection{First test case without control}
Its solution is not regular 
\begin{eqnarray}
u(t,x) &= & (1+t) \sin(\frac{x_2}{2}) \left \{ \begin{array}{ll}
                                              \sin \frac{x_1}{2}  &  \mbox{  for }  -2 \pi < x_1 < 0 \\
                                              \sin  \frac{x_1}{4}  &  \mbox{  for }  0 < x_1 < 2 \pi \\
                                              \end{array} \right.  \nonumber 
\end{eqnarray}
with 
\begin{eqnarray}
f_a(t,x) & =& \sin \frac{x_2}{2}  \left \{ \begin{array}{ll}
                     \sin \frac{ x_1}{2}  (1+ \frac{1+t}{4}) (\sin^2 x_1 + \sin^2 x_2)  &  \mbox{ for } -2 \pi < x_1 < 0 \\ 
                     \sin  \frac{ x_1}{4} (1+ \frac{1+t}{16}) (\sin^2 x_1 + 4 \sin^2 x_2)  &  \mbox{ for }  0   < x_1 < 2 \pi
                                          \end{array}
                                 \right .   \nonumber  \\
       &   & - \sin x_1 \sin x_2 \cos  \frac{x_2}{2}  \left \{ \begin{array}{ll}
                             \frac{1+t}{2} \cos \frac{ x_1}{2}  &  \mbox{ for } -2 \pi < x_1 < 0 \\ 
                             \frac{1+t}{4} \cos \frac{ x_1}{4}   &  \mbox{ for }  0   < x_1 < 2 \pi \\
                                           \end{array}
                                                   \right .  \nonumber  \\
c_a(t,x)  &= & 0,  \quad  b_a(t,x)= 0 \quad \sigma_a(t,x) =  \sqrt{2}  \left( \begin{array}{l}
                                                                             \sin x_1  \\
                                                                             \sin x_2  
                                                                 \end{array} 
                                                                 \right )  \nonumber 
\end{eqnarray}
On take  $Q = (0,1] \times [- 2 \pi, 2 \pi]^2$, the number of time step is equal to 400.
The error $er_n$  at a step $n$ is given in the infinite norm on the domain. The rate of convergence given is calculated as $\log(\frac{er_n}{er_{n-1}})$.
An error of $4e-4$ corresponds to the time discretization error of the scheme.
The solution is not regular but quadratic and cubic approximation give far better results than linear interpolation scheme. Besides it is difficult to get a steady rate of convergence.

\begin{table}
\small
\caption{Test case 1  } 
\label{Case1}

%\scalebox{0.8}{
 \begin{tabular}{|c|c|c|c|c|c|c|c|c|c|}
\hline
 & \multicolumn{3}{c|}{LINEAR} & \multicolumn{3}{c|}{QUADRATIC} & \multicolumn{3}{c|}{CUBIC}    \\
\hline
 Level   &  Err & Rate  & Time  &  Err & Rate  & Time  & Err & Rate  & Time         \\
\hline
4        &   0.3835  &    $-$      &   0     &  0.0497  &  $-$  &   0     &  0.0325    &  $-$   &   0   \\ 
5        &   0.4429  &   -0.14     &   0     &  0.0096  & 1.64  &   0     &  0.0163    &  0.7   &  0    \\ 
6        &   0.4301  &   0.03      &   0     &  0.0038  & 0.93  &   0     &  0.0046    &  1.25  &  0    \\
7        &   0.3376  &   0.24      &   1     &  0.0012  & 1.1   &   1     &  0.0004    &  2.3   &  1    \\
8        &   0.1794  &   0.63      &   3     &  0.0004  & 1.15  &   3     &          &          &       \\
9        &   0.0685  &   0.96      &   6     &          &       &         &          &          &       \\
10       &   0.0249  &   1.01      &   16    &          &       &         &          &          &        \\
11       &   0.0078  &   1.15      &   41    &          &       &         &          &          &        \\
12       &   0.0022  &   1.25      &   103   &          &       &         &          &          &  \\
13       &   0.0005  &   1.44      &   258   &          &       &          &          &          &  \\
\hline

\end{tabular} 
%}
\end{table}

\subsubsection{Control problem with a regular solution \cite{Jakobsen}, \cite{Zidani}}
The regular solution is given by 
\begin{eqnarray}
 u(t,x_1,x_2 ) = (\frac{3}{2} -t ) \sin x_1 \sin x_2 \nonumber 
\end{eqnarray}
Coefficients are given by 
\begin{eqnarray}
f_a(t,x) & =& (\frac{1}{2}-t) \sin x_1 \sin x_2  +  (\frac{3}{2}-t) \left[ \sqrt{ \cos^2 x_1 \sin^2 x_2 + \sin^2 x_1 \cos^2 x_2} \right. \nonumber \\
         &  & \left . - 2 \sin(x_1+x_2) \cos(x_1+x_2) \cos x_1 \cos x_2 \right]  \nonumber  \\
c_a(t,x)  &= & 0,  \quad  b_a(t,x)=  a \quad \sigma_a(t,x) =  \sqrt{2}  \left( \begin{array}{l}
                                                                             \sin(x_1+x_2)  \\
                                                                             \cos(x_1+x_2)  
                                                                 \end{array} 
                                                                 \right ), \nonumber  \\
   \mathop{A}& = &  \{ a \in \R^2 : a_1^2 + a_2^2 =1 \} \nonumber  
\end{eqnarray}
$Q = (0,1] \times [-\pi, \pi]^2$ and the number of time steps is equal to 800, the number of control equal to 400.
With  linear interpolation, the convergence is slow but exhibit a  steady rate. With quadratic or cubic, the
error is smaller but the convergence rate is more erratic. Oscillation for values around $0.0005$ are linked to the 
time discretization error.  The error and the convergence rates are calculated as in the previous case.

\begin{table}
%\small
\caption{Test case 2  } 
\label{Case2}

%\scalebox{0.8}{
 \begin{tabular}{|c|c|c|c|c|c|c|c|c|c|}
\hline
 &  \multicolumn{3}{c|}{LINEAR} & \multicolumn{3}{c|}{QUADRATIC} & \multicolumn{3}{c|}{CUBIC}    \\
\hline
 Level   &  Err & Rate  & Time  &   Err & Rate  & Time  &  Err & Rate  & Time         \\
\hline
4        &   0.4760  &    $-$      &   31       &  0.0154  &  $-$  &   32       &  0.0046    &  $-$   &  34  \\ 
5        &   0.3943  &   0.18      &   103      &  0.0072  &  0.74 &   106      &  0.0363    & -2.06   & 115   \\ 
6        &   0.2042  &   0.65      &   319      &  0.0032  & 0.80  &   320      &  0.0007    &  3.89  &  357  \\
7        &   0.0717  &   1.04      &   931      &  0.0104  & -1.17 &   968      &  0.0005    &  0.47   & 1047   \\
8        &   0.0216  &   1.20      &   2608     &  0.0005  & 2.97  &   2700     &    0.0005  &   $-$   &  2939     \\
9        &   0.0058  &   1.31      &   7109     &          &       &            &            &          &          \\
10       &   0.0012  &   1.52      &  18575     &          &       &            &            &          &        \\
11       &   0.0003  &   1.58      &  47032     &          &       &            &            &          &        \\
\hline
\end{tabular} 
%}
\end{table}

\section{Portfolio optimization}
\label{PortFolioSection}
All test cases extend or modify  some test cases taken from \cite{warin}.
We  report an application to the continuous-time portfolio optimization problem in financial mathematics. Let $\{S_t,t\in[0,T]\}$ be an It\^o  process modeling the price evolution of $n$ financial securities. The investor chooses an adapted process $\{\theta_t,t\in[0,T]\}$ with values in $\R^d$, where $\theta_t^i$ is the amount invested in  the $i-$th security held at time $t$. In addition, the investor has access to a non-risky security (bank account) where the remaining part of his wealth is invested. The non-risky asset $S^0$ is defined by an adapted interest rates process $\{r_t,t\in[0,T]\}$, i.e. $dS^0_t=S^0_tr_tdt$, $t\in[0,1]$. Then, the dynamics of the wealth process is described by:
 $$
 dX^\theta_t
 =
 \theta_t\cdot \frac{dS_t}{S_t} +(X^\theta_t-\theta_t\cdot \1)\frac{d S^0_t}{S^0_t}
 = 
 \theta_t\cdot \frac{dS_t}{S_t}+(X^\theta_t-\theta_t\cdot \1)r_tdt,
 $$
where $\1=(1,\cdots,1)\in\R^d$.
Let $\Ac$ be the collection of all adapted processes $\theta$ with values in $\R^d$, which are integrable with respect to $S$ and such that the process $X^\theta$ is uniformly bounded from below. Given an absolute risk aversion coefficient $\eta>0$, the portfolio optimization problem is defined by:
 \begin{eqnarray}\label{prob-portefeuille}
 v_0
 &:=&
 \sup_{\theta\in\Ac} \E\left[-\exp\left(-\eta X^\theta_T\right)\right].
 \end{eqnarray}
Under fairly general conditions, this linear stochastic control problem can be characterized as the unique viscosity solution of the corresponding HJB equation.  We shall first start by a two-dimensional example where an explicit solution of the problem is available. Then, we will present some results in a three, four and five dimensional situations.\\
At last in dimension 3, we will work on a less regular solution supposing that the investor is short of a at-the-money call at the initial date (strike $K$).
Then the function to optimize is
\begin{eqnarray}\label{call-prob-portefeuille}
 v_0
 &:=&
 \sup_{\theta\in\Ac} \E\left[-\exp\left(-\eta (X^\theta_T - (S_T -K)^{+} \right)\right].
 \end{eqnarray}

\subsubsection{A two dimensional problem}

Let $d=1$, $r_t=0$ for all $t\in[0,1]$, and assume that the security price process is defined by the Heston model \cite{heston}:
 \begin{eqnarray*}
 dS_t &=& \mu S_tdt +  \sqrt{Y_t} S_t dW_t^{(1)}  
 \\
 dY_t &=&  k (m-Y_t)dt + c\sqrt{Y_t} \left(\rho dW_t^{(1)}+\sqrt{1-\rho^2}dW_t^{(2)}\right),
 \end{eqnarray*}
where $W=(W^{(1)},W^{(2)})$ is a Brownian motion in $\R^2$.
In this context, it is easily seen that the portfolio optimization problem \reff{prob-portefeuille} does not depend on the state variable $s$. Given an initial state at the time origin $t$ given by $(X_t,Y_t)=(x,y)$, the value function $v(t,x,y)$ solves the HJB equation:
 \begin{equation}\label{hjb0}
 \begin{array}{rl}
 v(T,x,y) = - e^{-\eta x}
 ~\mbox{and}~
 0 
 =&\!\!
 - v_t -  k (m-y) v_y - \frac{1}{2}c^2 y v_{yy}\\
&\hspace*{1.7cm}
 -\Sup_{\theta\in\R} \Bigl(\frac12\theta^2yv_{xx}+\theta(\mu v_x+\rho cy v_{xy})\Bigr) 
 \\
 =&\!\!
 - v_t -  k (m-y) v_y - \frac{1}{2}c^2 y v_{yy}
 + \Frac{(\mu v_x + \rho cy v_{xy})^2} 
        {2 y v_{xx}}.
 \end{array}
 \nonumber
 \end{equation}
A quasi explicit solution of this problem was provided by Zariphopoulou \cite{zari}:
 \begin{eqnarray*}
 v(t,x,y)=-e^{-\eta x} \left\| \exp\left(-\frac{1}{2}\int_t^T \frac{\mu^2}{\tilde Y_s}ds\right)
                       \right\|_{\L^{1-\rho^2}}
 \end{eqnarray*}
where the process $\tilde Y$ is defined by
 \begin{eqnarray*}
 \tilde Y_t=y
 &\mbox{ and }&
 d\tilde Y_t = ( k (m-\tilde Y_t)- \mu c\rho)dt +  c\sqrt{\tilde Y_t} dW_t.
 \end{eqnarray*}
We take the same parameter as in \cite{warin} ($\eta=1$, $\mu =0.15$, $c  = 0.2$, $ k  = 0.1$, $m  = 0.3$, $Y_0 =  m$, $\rho = 0$). The initial portfolio value is equal to $1$, the maturity $T$ equal to $1$. The reference solution is $0.3534$.
In all calculations, the commands $\theta_t$ can take values in $[-1.5, 1.5]$ and the interval is discretized with 20 steps.
The resolution domain is $[-4,6]$ for the portfolio value and $[0.02,3]$ for the asset value.
The number of time step is equal to 200. 
 In table \ref{Case4BoundNoAdapt}, we give the results obtained without adaptation for the different interpolators using an exact boundary treatment where we suppose that the boundary solution is given by a constant investment in bonds giving a boundary value equal to $-e^{-1}$. With this resolution parameters, the spatially converged solution is 
$0.3535$.
The rate of convergence, not stable, is  not given.
In table \ref{Case4NoBoundNoAdapt}, the same calculation is achieved without boundary points in the approximation.
Using an exact boundary treatment, quadratic and cubic converge faster than linear interpolation. Cubic is not superior to quadratic.
Using extrapolated boundary, where no theoretical convergence is available, linear interpolation gives the same results as in the 
case of exact boundary treatment. Quadratic  and cubic converges faster when no boundary conditions is imposed  certainly meaning that the boundary condition imposed is not optimal.
\begin{table}
%\small
\caption{Test case 3 :  portfolio optimization in dimension 2, no adaptation , exact boundary treatment } 
\label{Case4BoundNoAdapt}

%\scalebox{0.8}{
 \begin{tabular}{|c|c|c|c|c|c|c|}
\hline
 &  \multicolumn{2}{c|}{LINEAR} & \multicolumn{2}{c|}{QUADRATIC} & \multicolumn{2}{c|}{CUBIC}    \\
\hline
 Level   & Solution &  Time  &  Solution & Time  &  Solution  & Time         \\
\hline
6        &  -0.3678  &    5       &  -0.3622  &   5     &   -0.3629      &   5   \\
7        &  -0.3670  &    14      &  -0.3433  &   15    &   -0.3360      &   16    \\
8        &  -0.3565  &    40      &  -0.3555  &   40    &  -0.3565      &   43       \\
9        &  -0.3550  &    105     &  -0.3533  &   109   &  -0.3531    &   116       \\
10       &  -0.3539  &    274     &  -0.3535  &   283   &   -0.3535   &   304        \\
11       &  -0.3536  &    700     &            &         &          &           \\
12       &  -0.3535  &    1757    &            &         &          &           \\
\hline
\end{tabular} 
%}
\end{table}
\begin{table}
%\small
\caption{Test case 3 : portfolio optimization in dimension 2, no adaptation , extrapolated boundary treatment } 
\label{Case4NoBoundNoAdapt}

%\scalebox{0.8}{
 \begin{tabular}{|c|c|c|c|c|c|c|}
\hline
  & \multicolumn{2}{c|}{LINEAR} & \multicolumn{2}{c|}{QUADRATIC} & \multicolumn{2}{c|}{CUBIC}    \\
\hline
 Level   & Solution &  Time  &  Solution & Time  &  Solution  & Time         \\
\hline
6        &  -0.3678  &    3       &   -0.3576  &   3    &  -0.0.3575      &   4   \\
7        &  -0.3668  &    8       &   -0.3522  &   9    &  -0.3519        &   9    \\
8        &  -0.3579  &    24      &   -0.3536  &   24   &  -0.3536        &   26        \\
9        &  -0.3551  &    64      &   -0.3535  &  67    &  -0.3535        &  71       \\
10       &  -0.3539  &    173     &            &        &                 &          \\
11       &  -0.3536  &    451     &            &        &                 &           \\
12       &  -0.3535  &    1145    &            &        &                 &           \\
\hline
\end{tabular} 
%}
\end{table}
The same test case can be achieved starting with an initial level of 5 and local adaptation.
Refinement is only allowed at the center of the domain $[-2.5,4.5] \times [0.05,1.54]$ limiting the maximal level of refinement to $12$.
The use of adaptation permits to reduce the number of points used by focusing on the region of interest.
It is especially effective with the quadratic or cubic interpolation.

\begin{table}
%\small
\caption{Test case 3 :  portfolio optimization in dimension 2, adaptation , exact boundary treatment, initial level 5 } 
\label{Case4BoundAdapt}
%\scalebox{0.8}{
 \begin{tabular}{|c|c|c|c|c|c|c|}
\hline
 & \multicolumn{2}{c|}{LINEAR} & \multicolumn{2}{c|}{QUADRATIC} & \multicolumn{2}{c|}{CUBIC}    \\
\hline
 Precision   & Solution &  Time  &  Solution & Time  &  Solution  & Time         \\
\hline
0.001       &  -0.3593  &    17      &  -0.3495   &   13     &  -0.3535    &   10   \\
0.00025     &  -0.3553  &    39      &   -0.3545  &   25     &   -0.3535    &   18    \\
6.25e-05    &  -0.3542  &    84      &   -0.3536  &   50     &              &        \\
1.56e-05    &  -0.3537  &    157     &   -0.3535  &   90     &             &        \\
\hline
\end{tabular} 
%}
\end{table}

\begin{table}
% \small
\caption{Test case 3 :  portfolio optimization in dimension 2, adaptation , extrapolated  boundary treatment, initial level 5 } 
\label{Case4NoBoundAdapt}
%\scalebox{0.8}{
 \begin{tabular}{|c|c|c|c|c|c|c|}
\hline
 & \multicolumn{2}{c|}{LINEAR} & \multicolumn{2}{c|}{QUADRATIC}  & \multicolumn{2}{c|}{CUBIC}    \\
\hline
 Precision   & Solution &  Time  &   Solution & Time   &  Solution & Time         \\
\hline
0.001       &  -0.3581  &    10      &  -0.3537     &   4   &  -0.3540      &   3   \\
0.00025     &  -0.3556  &    25      &  -0.3536     &   7   &  -0.3536      &   6 \\                         
6.25e-05    &  -0.3542  &    54      &  -0.3535     &  16   &  -0.3535       &  13 \\                          
1.56e-05    &   -0.3538 &    101     &              &       &               &           \\       \hline
\end{tabular} 
%}
\end{table}
\begin{figure}[h]
\centering
\includegraphics[width=5cm]{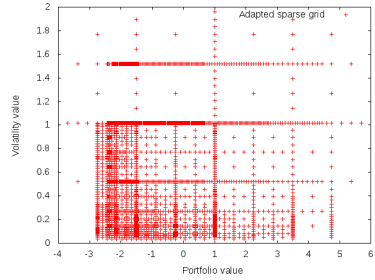}
\caption{Example of adapted meshes in dimension 2  (test case 3 with extrapolated boundary)}
\label{AdaptedMeshNoBound}
\end{figure}
\subsubsection{A three dimensional problem}
We now let $n=1$ , and we assume that the interest rate process is defined by the Ornstein-Uhlenbeck process:
 \begin{eqnarray*}
 dr_t &=& \kappa(b-r_t)dt+\zeta dW^{(0)}_t.
\end{eqnarray*}
The security has the same dynamic as in the previous case. We assume that all correlations are equal to 0.
The optimization problem is still given by equation \reff{prob-portefeuille}.
In this case the value function $v(t,x,r,y)$ satisfies the HJB equation :
\begin{align}
 0&=-v_t - (\mathbf{L}^r +\mathbf{L}^Y)v - r x v_x 
 %\nonumber\\
 % & 
-\sup_{\theta} 
      \left\{\theta (\mu\!-\!r )v_x
             \!+\! \frac{\theta^2}{2}y v_{xx}
      \right\} \nonumber
 \end{align}
where
 \begin{eqnarray*}
 \mathbf{L}^r v=\kappa(b-r)v_r+\frac12\zeta^2 v_{rr},
 &&
 \mathbf{L}^Y v= k\left(m -y\right)v_{y} +  \frac12 c^2  yv_{y y},
 \end{eqnarray*}
We take the same parameter as in the two dimensional case. As for the interest rate model we take $b=0.07$, $r_0= b$, $\zeta= 0.3$. The initial asset values and wealth values
are equal to $1$. The discretization domain  is $[-4,10]\times [-0.2,0.5] \times [0.02,2] $.
The number of time steps is still equal to $200$, the number of commands tested equal to $50$.
We use extrapolated boundary condition and  results are given in table \ref{Case5NoBoundNoAdapt}. Calculation are achieved
with parallelization on  bi processor with 24 cores.
During adaptation the maximal refinement level is fixed to 11. 
Adaptation appears to be very effective in this case trimming the cost of calculation in all the cases.
\begin{table}
%\small
\caption{Test case 4 : portfolio optimization in dimension 3, no adaptation , extrapolated boundary treatment } 
\label{Case5NoBoundNoAdapt}
%\scalebox{0.8}{
 \begin{tabular}{|c|c|c|c|c|c|c|}
\hline
 & \multicolumn{2}{c|}{LINEAR} & \multicolumn{2}{c|}{QUADRATIC}  & \multicolumn{2}{c|}{CUBIC}    \\
\hline
 Level   & Solution &  Time  &  Solution & Time   &  Solution  & Time       \\
\hline
5        &  -0.3458  &    0        &  -0.3387     &    1       &     -0.3403   &  1              \\
6        &  -0.3441  &    4        &  -0.3394     &    3       &      -0.3397  &   3              \\
7        &  -0.3437  &    13       &  -0.3379     &    13      &      -0.3378   &    14             \\
8        &  -0.3444  &    46       &  -0.3384     &    48      &      -0.3384   &  51                \\
9        &  -0.3396  &    158      &  -0.3383     &    163     &      -0.3383   &   176              \\
10       &  -0.3387  &    517      &  -0.3382     &    543     &      -0.3382  &     570       \\
11       &  -0.3384  &    1655     &  -0.3383     &    1670    &      -0.3383   &  1787          \\
\hline
\end{tabular} 
%}
\end{table}
\begin{table}
%\small
\caption{Test case 4 : portfolio optimization in dimension 3, adaptation , extrapolated boundary treatment } 
\label{Case5NoBoundAdapt}
%\scalebox{0.8}{
 \begin{tabular}{|c|c|c|c|c|c|c|}
\hline
 & \multicolumn{2}{c|}{LINEAR} & \multicolumn{2}{c|}{QUADRATIC} & \multicolumn{2}{c|}{CUBIC}   \\
\hline
 Precision   & Solution & Time &   Solution & Time  &   Solution & Time        \\
\hline
1e-3         &  -0.3430   &    15     & -0.3380    &   10  &     -0.3385 & 9   \\
0.00025      &  -0.3401   &    38     & -0.3381    &   24  &     -0.3383  & 23  \\
6.25e-05     &  -0.3390   &    90     & -0.3381    &   53  &     -0.3381 & 54   \\
1e-5         &  -0.3384   &   162     & -0.3381    &  105  &     -0.3382 & 97       \\
\hline  
\end{tabular} 
%}
\end{table}
\subsubsection{A four dimension example}
Now take the same problem as before but modify the dynamic of the asset using a CEV-SV model (see \cite{lord} for a presentation of this model) :
\begin{eqnarray*}
 dS_t
 &=& 
 \mu S_tdt + \sigma \sqrt{Y_t} {S_t}^{\beta} dW^{(1)}_t,
 \\
 dY_t 
 &=&  k\left(m -Y_t\right)dt +  c  \sqrt{Y_t} dW^{(2)}_t 
 \end{eqnarray*}
The optimization problem is still given by equation \reff{prob-portefeuille}.
In this case the value function $v(t,x,r,s,y)$ satisfies the HJB equation :
 \begin{align}
 0&=-v_t - (\mathbf{L}^r +\mathbf{L}^Y +\mathbf{L}^{S})v - r x v_x 
 \nonumber\\
  & -\sup_{\theta} 
      \left\{\theta (\mu\!-\!r)v_x
             \!+\!\theta \sigma^2 y s^{2 \beta-1} v_{xs}
             \!+\!\frac12 \theta^2\sigma^2 y s^{2\beta-2}
      \right\} \nonumber
 \end{align}
where
\begin{eqnarray*}
 \mathbf{L}^{S}v
 &=&
 \mu sv_{s}+\frac12\sigma^2 s y v_{ss}.
 \end{eqnarray*}
Take $\eta=1$,  $\mu = 0.10$, $\sigma =0.3$, $\beta =0.5$ for the asset,
$ k =0.1$, $m = 1.$, $c = 0.1$ for the diffusion process of the asset. The maturity $T$ is still equal to $1$.
Parameters for the interest rate are the same as in the three dimension case.
The resolution domain is  $[-5,10]\times [-0.2,0.5] \times [0.02,5] \times [0.02,5] $.
The refinement domain is defined on $[-3.5,8.5]\times [-0.13,0.43] \times [0.5, 4.5] \times [0.5,4.5] $
calculation are achieved on 24 cores. The maximal refinement level is fixed to $12$.
In all the cases, the convergence appears to be rather difficult and a solution with more than 3 digits is hard to find.
\begin{table}
%\small
\caption{Test case 5 : portfolio optimization in dimension 4, no adaptation , extrapolated boundary treatment } 
\label{Case6NoBoundNoAdapt}
%\scalebox{0.8}{
 \begin{tabular}{|c|c|c|c|c|c|c|c|}
\hline
 & \multicolumn{2}{c|}{LINEAR} & \multicolumn{2}{c|}{QUADRATIC}& \multicolumn{2}{c|}{CUBIC}    \\
\hline
 Level   & Solution &  Time  &  Solution & Time &   Solution & Time        \\
\hline
5        &  -0.3542  &    4       &  -0.3315 &   4        &  -0.3329   &    4     \\
6        &  -0.3435  &    20      &  -0.3354 &   20       &  -0.3371   &    21      \\
7        &  -0.3453  &    99      &  -0.3358 &   102      &  -0.3361   &     108     \\
8        &  -0.3443  &    441     &  -0.3364 &   452      &  -0.3384|  &     481       \\
9        &  -0.3382  &    1832    &  -0.3360 &   1877     &  -0.3359   &    2017       \\
10       &  -0.3358  &    7336    &  -0.3360  &   7478     &  -0.3360    &     7915     \\
11       &  -0.3357  &    28210   & -0.3352  &  28870     &    -0.3353 &   30143          \\
12       &   -0.3358        &   101921   & -0.3361  &  104950    &   -0.3360  &  108000         \\
\hline 
\end{tabular} 
%}
\end{table}
\begin{table}
%\small
\caption{Test case 5 : portfolio optimization in dimension 4, adaptation , extrapolated boundary treatment, initial level equal to 6 } 
\label{Case6NoBoundAdapt}
%\scalebox{0.8}{
 \begin{tabular}{|c|c|c|c|c|c|c|c|}
\hline
 & \multicolumn{2}{c|}{LINEAR} & \multicolumn{2}{c|}{QUADRATIC}& \multicolumn{2}{c|}{CUBIC}    \\
\hline
 Precision   & Solution &  Time  &  Solution & Time &   Solution & Time        \\
\hline
0.001          &  -0.3415  &    610      &  -0.3373 &   302       &  -0.3367   &    289       \\
0.00025        &  -0.3369  &    2420       & -0.3349 &   1471      &  -0.3353   &   1458      \\
6.25e-05       &  -0.3356  &    7703      &  -0.3353 &   5701      &  -0.3353   &    5784     \\
1.56e-05       &  -0.3355  &    14504      &  -0.3350 &   12716      &  -0.3349  &   12860     \\
\hline 
\end{tabular} 
%}
\end{table}
\subsubsection{A five dimensional example}

We now let $n=2$, and we assume that the interest rate process is defined by the Ornstein-Uhlenbeck process:
 \begin{eqnarray*}
 dr_t &=& \kappa(b-r_t)dt+\zeta dW^{(0)}_t.
 \end{eqnarray*}
While the price process of the second security is defined by an Heston model, the first security's price process is defined by a CEV-SV model :
 \begin{eqnarray*}
 dS^{(i)}_t
 &=& 
 \mu_i S^{(i)}_tdt + \sigma_{i} \sqrt{Y^{(i)}_t} {S^{(i)}_t}^{\beta_i} dW^{(i,1)}_t,~~\beta_2=1,
 \\
 dY^{(i)}_t 
 &=&  k_i\left(m_i -Y^{(i)}_t\right)dt +  c_i  \sqrt{Y^{(i)}_t} dW^{(i,2)}_t 
 \end{eqnarray*}
where $\left(W^{(0)},W^{(1,1)},W^{(1,2)},W^{(2,1)},W^{(2,2)}\right)$ is a Brownian motion in $\R^5$, and for simplicity we considered a zero-correlation between the security price process and its volatility process. 

Since $\beta_2=1$, the value function of the portfolio optimization problem \reff{prob-portefeuille} does not depend on the $s^{(2)}-$variable. 
Given an initial state $(X_t,r_t,S^{(1)}_t$ $,Y^{(1)}_t,Y^{(2)}_t)=(x,r,s_1,y_1,y_2)$ at the time origin $t$, the value function $v(t,x,r,$ $s_1,y_1,y_2)$ satisfies the HJB equation:
 \begin{align}
 0&=-v_t - (\mathbf{L}^r +\mathbf{L}^Y +\mathbf{L}^{S^1})v - r x v_x 
 \nonumber\\
  & -\sup_{\theta_1,\theta_2} 
      \left\{\theta \cdot(\mu\!-\!r\1)v_x
             \!+\!\theta_1 \sigma_1^2 y_1 s_1^{2 \beta_1-1} v_{xs_1}
             \!+\!\frac12(\theta_1^2\sigma_1^2 y_1 s_1^{2\beta_1-2}
             \!+\!\theta_2^2\sigma_2^2y_2)v_{xx}
      \right\} \nonumber
 \end{align}
where
 \begin{eqnarray*}
 \mathbf{L}^r v=\kappa(b-r)v_r+\frac12\zeta^2 v_{rr},
 &&
 \mathbf{L}^Y v=\sum_{i=1}^2k_i\left(m_i -y_i\right)v_{y_i} +  \frac12 c_i^2  y_iv_{y_i y_i},
 \end{eqnarray*}
 \begin{eqnarray*}
 \mbox{and}~~\mathbf{L}^{S^1}v
 &=&
 \mu_1 s_1v_{s_1}-\frac12\sigma_1^2 s_1 y_1 v_{s_1s_1}.
 \end{eqnarray*}
We take the following parameters (as in \cite{warin}) :
$\eta=1$,  $\mu_1 = 0.10$, $\sigma_1 =0.3$, $\beta_1 =0.5$ for the first asset,
$ k_1 =0.1$, $m_1 = 1.$, $c_1 = 0.1$ for the diffusion process of the first asset. The second asset is defined by the same parameters as in the two dimensional example: $\mu_2=0.15$, $c_2=0.2$, $m=0.3$ and $Y^{(2)}_0=m$. As for the interest rate model we take $b=0.07$, $r_0= b$, $\zeta= 0.3$. The initial asset values and wealth values
are equal to $1$, the maturity is equal to $1$.
The number of time steps is equal to $200$.\\
The resolution domain is  $[-5,10]\times [-0.2,0.5] \times [0.02,5] \times [0.015,7] \times [0.15,7] \times [0.04,2.1]$.
The space of control is in dimension 2 : $(\theta_1,\theta_2)$  are taken in $[-1;5, 1.5] \times [-1.5,1.5]$. The classical way to find the optimal control at a given point for a given date consists in
discretizing the command with a thin mesh and to test all the commands to get the optimal one. 
We propose to use a sparse grid in dimension 2 to represent the space of commands and we interpolate the function obtained to estimate this optimal command on a very thin grid. Once the optimal command is obtained, it is used to re estimate the value function.
This approach has been tested  with the cubic interpolator and results are given in table \ref{Case7_command}. Results are obtained with 384 cores. 

\begin{table}
\small
\caption{Sparse grid for commands : 'Level' indicates the level of the sparse grid used to approximate to value function. 'Discretization for commands' indicate the discretization used to interpolated the commands. In the case of use of sparse grids to discretize commands, the level of the sparse grid is given by 'Level C'. The optimal command is calculated on the thin grid $64 \times 64$ } 
\label{Case7_command}
 \begin{tabular}{|c|c|c|c|c|c|c|}
\hline
  Discretization for commands & \multicolumn{2}{c|}{Full grid $64 \times 64$} & \multicolumn{2}{c|}{Level C = $4$} & \multicolumn{2}{c|}{Level C = $5$}    \\
\hline\hline
 Level for solution  & Solution &  Time  &  Solution & Time &   Solution & Time        \\
\hline
  9      &  -0.2998  &    5295      &  -0.2998 &   89       &  -0.2998   &    173      \\
 10      &  -0.3083  &    29046       & -0.3081 &   424      &  -0.3083   &   886      \\
\hline
\end{tabular} 
\end{table}
The results obtained for the valorization in dimension 5 for linear, quadratic and cubic estimator are given in table
\ref{Case7NoBoundNoAdapt} with a maximal discretization level equal to $10$.  Calculation time are obtained for 384 cores.
Results with adaptation are given in table \ref{Case7NoBoundAdapt}.
\begin{table}
%\small
\caption{Test case 6 : portfolio optimization in dimension 5, no adaptation , extrapolated boundary treatment } 
\label{Case7NoBoundNoAdapt}
%\scalebox{0.8}{
 \begin{tabular}{|c|c|c|c|c|c|c|}
\hline
  & \multicolumn{2}{c|}{LINEAR} & \multicolumn{2}{c|}{QUADRATIC} & \multicolumn{2}{c|}{CUBIC}    \\
\hline
 Level   & Solution &  Time  &  Solution & Time  &  Solution  & Time         \\
\hline
6        & -0.7167   &  50       & -0.2889  &  51  &    -0.2933   &  52    \\
7        & -0.3326   &  230      & -0.3035   & 227   &    -0.3044   &  233   \\
8        & -0.2980   &  1032     &  -0.3124   & 1047   &    -0.3132   &  1091  \\
9        & -0.3134   &  4641     &  -0.3092  &  4716  &    -0.3091   &  4980  \\
10       & -0.3112   &  21263    &  -0.3089  &  21238   &    -0.3089   &  22500    \\
\hline
\end{tabular} 
%}
\end{table}
\begin{table}
%\small
\caption{Test case 6 : portfolio optimization in dimension 5 , adaptation , extrapolated boundary treatment, initial level equal to 7 } 
\label{Case7NoBoundAdapt}
%\scalebox{0.8}{
 \begin{tabular}{|c|c|c|c|c|c|c|c|}
\hline
 & \multicolumn{2}{c|}{LINEAR} & \multicolumn{2}{c|}{QUADRATIC}& \multicolumn{2}{c|}{CUBIC}    \\
\hline
 Precision   & Solution &  Time  &  Solution & Time &   Solution & Time        \\
\hline
0.001          &-0.3392 &   1311     & -0.3085  &  1166    &   -0.3091 &   1170    \\
0.00025        &-0.3116 &   2874     &  -0.3098 &   2212   &   -0.3097 &   2192     \\
6.25e-05       &-0.3092 &   5667     & -0.3101  &   3817   &   -0.3105 &   3738   \\
1.56e-05       & -0.3101 &  10115    &  -0.3095 &   6396   &   -0.3095 &   6262   \\
\hline 
\end{tabular} 
%}
\end{table}

\subsection{A less regular three dimension problem}
In this problem, we have only one security  defined by a CEV-SV models :
 \begin{eqnarray*}
 dS_t
 &=& 
 \mu S_tdt + \sigma \sqrt{Y_t} {S_t}^{\beta} dW^{(1)}_t,
 \\
 dY_t 
 &=&  k \left(m -Y_t\right)dt +  c  \sqrt{Y_t} dW^{(2)}_t 
 \end{eqnarray*}
The investor short of a call option wants to optimize his portfolio following equation \reff{call-prob-portefeuille}.
Given an initial state $(X_t,S_t,Y_t)=(x,s,y)$ at the time origin $t$, the value function $v(t,x,s,y)$ satisfies the HJB equation:
 \begin{align}
 0&=-v_t - (\mathbf{L}^Y +\mathbf{L}^{S})v 
 \nonumber\\
  & -\sup_{\theta} 
      \left\{\theta \cdot \mu v_x
             \!+\!\theta \sigma ^2 y s^{2 \beta-1} v_{xs}
             \!+\!\frac12(\theta^2\sigma^2 y s^{2\beta-2}
      \right\} \nonumber
 \end{align}
where
 \begin{eqnarray*}
 \mathbf{L}^Y v= \left(m -y\right)v_{y} +  \frac12 c^2  y v_{y y},
 \end{eqnarray*}
 \begin{eqnarray*}
 \mbox{and}~~\mathbf{L}^{S}v
 &=&
 \mu s v_{s}-\frac12\sigma^2 s y v_{ss}.
 \end{eqnarray*}
The asset parameters are given by   $\mu = 0.10$, $\sigma =0.3$, $\beta =0.5$. The diffusion asset has the parameters  $ k =0.1$, $m = 1.$, $c = 0.1$.
The commands $\theta_t$ can take values in $[-1.5, 1.5]$ and the interval is discretized with 50 steps.
We solve the problem with a number of time step equal to $200$.
We take the solution calculated with extrapolated boundaries with the linear, quadratic and cubic  interpolator. We use 24 cores for the different calculations.
As it can be seen in table \reff{Case8NoBoundNoAdapt} convergence is slower than in the more regular cases.
As it can be seen in table \reff{Case8NoBoundAdapt}, the adaptation gives good results but its interest for the cubic case is not obvious. 
\begin{table}
%\small
\caption{Test case 7 : portfolio optimization  short of a call option in dimension 3 , no adaptation , extrapolated boundary treatment } 
\label{Case8NoBoundNoAdapt}
%\scalebox{0.8}{
 \begin{tabular}{|c|c|c|c|c|c|c|}
\hline
 & \multicolumn{2 }{c|}{LINEAR} & \multicolumn{2}{c|}{QUADRATIC}& \multicolumn{2}{c|}{CUBIC}    \\
\hline
 Level   & Solution &  Time  &  Solution & Time   &  Solution & Time       \\
\hline
7        &  0.8909   &     13     &  -0.5936 &  13     & -0.3173  &  14   \\
8        &  -0.3373  &    46      &  -0.4313 &   47    &  -0.4069 &  51     \\
9        &  -0.2220  &     157    &  -0.4294 &   161   &  -0.4340 &  171   \\
10       &  -0.4101  &    521     &  -0.4327 &   524   & -0.4325  &  566          \\
11       & -0.4208   &    1629    &  0.4311  &   1640  & -0.4309  &  1760            \\
12       & -0.4299   &    4856    &  -0.4310 &  4981   & -0.4309  &   5325            \\
13       &  -0.4303  &    14233   &  -0.4309 &  14527  &  -0.4309 & 15536    \\
\hline
\end{tabular} 
%}
\end{table}

\begin{table}
%\small
\caption{Test case 7 : portfolio optimization  short of a call option in dimension 3 , adaptation with initial level equal to 9, extrapolated boundary treatment } 
\label{Case8NoBoundAdapt}
%\scalebox{0.8}{
 \begin{tabular}{|c|c|c|c|c|c|c|}
\hline
 & \multicolumn{2 }{c|}{LINEAR} & \multicolumn{2}{c|}{QUADRATIC}& \multicolumn{2}{c|}{CUBIC}    \\
\hline
 Precision   & Solution &  Time  &  Solution & Time   &  Solution & Time       \\
\hline
 0.001           & -0.3904   & 2452    &  -0.4329 & 1515   & - 0.4308  & 1428 \\
 0.00025         & -0.4181   & 3198    &  -0.4311 & 2123   & - 0.4311 &  2002   \\
 6.25e-05        & -0.4282   & 3910    &  -0.4309 & 2818   & - 0.4309  &  2663   \\
 1.56e-05        & -0.4294  &  4249    & &  &&  \\
 3.9e-06         & -0.4299  &  4541    & &  &&  \\
\hline
\end{tabular} 
%}
\end{table}

\end{document}